\documentclass[a4paper]{article}

\usepackage{amsmath,amsfonts,amssymb,amsthm}
\usepackage[english]{babel}
\newcommand{\beq}{\begin{equation}}
\newcommand{\bey}{\begin{eqnarray}}
\newcommand{\beyy}{\begin{eqnarray*}}
\newcommand{\eeq}{\end{equation}}
\newcommand{\eey}{\end{eqnarray}}
\newcommand{\eeyy}{\end{eqnarray*}}

\newcommand{\N}{{\mathbb N}}

\newcommand{\myspace}{\qquad\qquad\qquad}

\DeclareMathOperator{\wlim}{w-lim}

\newcommand{\cA}{{\mathcal A}}

\newcommand{\cD}{{\mathcal D}}

\newcommand{\cL}{{\mathcal L}}
\newcommand{\cM}{{\mathcal M}}

\newcommand{\cT}{{\mathcal T}}

\newtheorem{theorem}{Theorem}[section]
\newtheorem{lemma}[theorem]{Lemma}
\newtheorem{proposition}[theorem]{Proposition}
\newtheorem{remark}[theorem]{Remark}
\newtheorem{remarks}[theorem]{Remarks}
\newtheorem{assumption}[theorem]{Assumption}
\newtheorem{assumptions}[theorem]{Assumptions}

\newtheorem{corollary}[theorem]{Corollary}
\newtheorem{question}[theorem]{Question}
\newtheorem{claim}[theorem]{Claim}
\newtheorem{problem}[theorem]{Problem}
\numberwithin{equation}{section}

\begin{document}

% Top matter

\title{\bf \Large A theory of the infinite horizon LQ-problem for composite systems of 
PDEs with boundary control\footnote{The research of P.~Acquistapace was partially supported 
by the Universit\`a di Pisa.
The research of F.~Bucci was partially supported by the Universit\`a degli Studi di Firenze under the Project {\em Calcolo delle variazioni e teoria del controllo} and by the Italian MIUR 
under the PRIN~2009KNZ5FK Project ({\em Metodi di viscosit\`a, geometrici e di controllo per modelli diffusivi nonlineari}).
The research of I.~Lasiecka was partially supported by the NSF~Grant DMS-0606682 and AFOSR~Grant 
FA9550-09-1-0459.}
}

\author{
Paolo Acquistapace\\
\footnotesize{Universit\`a di Pisa}\\
\footnotesize{Pisa, Italy}\\
%\footnotesize{\tt acquistp@unipi.it} 
%
\and 
Francesca Bucci\\
\footnotesize{Universit\`a degli Studi di Firenze}\\
\footnotesize{Firenze, Italy}\\
%\small{\tt francesca.bucci@unifi.it}
%
\and
Irena Lasiecka\\
\footnotesize{University of Virginia}\\
\footnotesize{Charlottesville VA, U.S.A.}\\
%\small{\tt il2v@virginia.edu}
}

% \date{\today}
\date{\empty}

\maketitle

\begin{abstract}
We study the infinite horizon Linear-Quadratic problem and the associated algebraic Riccati equations for systems with unbounded control actions. The operator-theoretic context is motivated by composite systems of Partial Differential Equations (PDE) with boundary or point control. Specific focus is placed on systems of coupled hyperbolic/parabolic PDE with an overall `predominant' hyperbolic character, such as, e.g., some models for thermoelastic or fluid-structure interactions.
While unbounded control actions lead to Riccati equations with unbounded (operator) coefficients,
unlike the parabolic case solvability of these equations becomes a major issue, owing to the lack 
of sufficient regularity of the solutions to the composite dynamics.
In the present case, even the more general theory appealing to estimates of the singularity displayed by the kernel which occurs in the integral representation of the solution to the control system fails.
A novel framework which embodies possible hyperbolic components of the dynamics has been introduced
by the authors in 2005, and a full theory of the LQ-problem on a finite time horizon has been developed. The present paper provides the infinite time horizon theory, culminating in well-posedness of the corresponding (algebraic) Riccati equations. New technical challenges are encountered and new tools are needed, especially in order to pinpoint the differentiability of the optimal solution.
The theory is illustrated by means of a boundary control problem arising in thermoelasticity. 
\end{abstract}

%----------------------------------------------------------------------------
\section*{Introduction}
%-----------------------------------------------------------------------------
{\bf Historical background and motivation. }
The theory of the optimal control problem with coercive quadratic functionals for abstract 
linear systems $y'=Ay+Bu$---in spaces of finite or infinite dimensions---is intrinsically 
linked to well-posedness of the corresponding algebraic/differential Riccati equations.
As it is well known, by solving these celebrated equations one obtains the operator (a matrix, in the finite dimensional case) which occurs in the feedback representation of the optimal control as well as in the quadratic form which yields the optimal cost.

In an infinite dimensional setting, if $B$ is a {\em bounded} control operator from the control space $U$ to the state space $Y$, then a complete theory of the Riccati equations has been developed, where the analogs of all the key properties known in the classical finite dimensional context hold true.
Namely, the unique solution to the Riccati equation provides the optimal cost operator, the feedback representation of the optimal control holds true, and the closed loop equation yields the synthesis of the optimal control; see \cite{bala} for pioneering results in this direction, \cite{barbuc}, \cite{bddm} and \cite{curtain-zwart}. 

When it comes to the abstract representation of initial/boundary value problems for evolutionary Partial Differential Equations (PDE) in a bounded domain $\Omega\subset \mathbb{R}^n$, with control action exercised on the boundary $\partial\Omega$ or pointwise in the interior of the domain, the major technical challenges come from the {\em unboundedness} of the control operator $B$ which naturally arises from the modeling of the PDE problem. 
This, in turn, results in the possible unboundedness of the gain operator $B^*P$ which occurs 
in the quadratic term of the Riccati equations. 
In fact, this operator may be even {\em not densely defined}.
Indeed, this `pathology' is shown to happen even in the case of simple hyperbolic equations
with point control; see \cite{weiss-zwart}, \cite{triggiani}.

In this respect, the regularity properties of the pair $(A,B)$ which describes the free dynamics
and the control action or, more specifically, of $e^{At}B$ (where $e^{At}$ is the semigroup governing the free dynamics) are absolutely central.
And in fact, historically speaking, the first extensions of the LQ-problem theory for systems 
with bounded control operator to the present setting 
% (where $B$ is not bounded form $U$ to $Y$) 
pertained to parabolic-like dynamics, where the underlying semigroup $e^{At}$ is {\em analytic}.
(Canonical illustrations are the heat equation with boundary or point control, as well as structurally damped plate equations, and certain thermoelastic systems.)
Analyticity made it possible the full development of a theory of the Riccati equations, with 
unbounded gains as well as without  additional regularity properties of the optimal cost operator.
These results date back to the beginning of the eighties (\cite{las-trig-analytic});
see \cite{las-trig-lecturenotes} for a concise treatise, or, alternatively, the 
extended monographs \cite{bddm} and \cite{las-trig-encyclopedia}. 
The analytic theory of Riccati equations with {\em non-autonomous} coefficients has been 
developed in \cite{a1}, \cite{a2}, \cite{a3,a4}.

In contrast, as it was pointed out already, the lower regularity exhibited by $e^{At}B$
in the case of hyperbolic-like equations with boundary/point control is in general insufficient 
to guarantee the sought-after properties recorded above. 
(Canonical examples are provided by PDE which display some kind of `wave propagation',
even in one dimension, under unbounded control actions.)
Thus, a requirement of an appropriate {\em smoothing effect} of the observation operator is 
called for, in order to establish well-posedness of the differential Riccati equations. 
See, again, \cite{las-trig-lecturenotes}, \cite{bddm}, \cite{las-trig-encyclopedia},
and the more recent \cite{barbu-las-trig}. 

\smallskip
The first break with analyticity and additional assumptions on the observation operator was introduced by G.~Avalos in 1996 in the study of PDE models that arise in structural acoustics; 
see \cite{avalos-96,avalos-las-1996}.
In these problems the equation for the acoustic waves is coupled on an interface with a 
parabolic-like equation describing the structural vibrations of an elastic wall and 
therefore the systems under consideration comprise both hyperbolic and parabolic dynamics. 
The study carried out in \cite{avalos-las-1996} pinpoints and establishes a fundamental ({\em singular}) estimate for the norm of the operator $e^{At}B$ in the space $\cL(U,Y)$;
this is recognized as the actual mathematical property which {\em by itself}---in the absence of analyticity---is sufficient to ensure boundedness of the gain operator. 
It is worth noting that while the aforesaid `singular estimates' are intrinsic to control systems whose free dynamics is governed  % ruled 
by an analytic semigroup $e^{At}$, in the case of general $C_0$-semigroups these estimates correspond to suitable (non trivial) interior regularity results for the solutions of the {\em uncontrolled} system of PDE.   

This path of investigation has been followed later on, bringing about a theory of the LQ-problem
on both a finite and infinite time horizon, while other significant PDE systems yielding `singular estimates' have been discovered and analyzed; see \cite{cbms,las-trento}, \cite{las-trig-se-1,las-trig-se-2}, and \cite{las-tuff-2008} (dealing with the Bolza problem).
It turned out that the class of control systems characterized by this property---although inspired by the model for acoustic-structure interactions studied in \cite{avalos-las-1996}---covers a variety of systems of coupled hyperbolic/parabolic PDE, including different structural acoustic models, thermoelastic plate models, composite (sandwich) beams models, and others; see, e.g.,
\cite{cbms}, \cite[Section~4]{las-trig-se-2}, \cite[Lecture~III, Part~II]{las-trento}, \cite{bucci-las-thermo}, \cite{las-tuff-2009}.
(See also \cite{bucci-jee} for further applications of singular estimates to the study of 
semilinear evolution equations with nonlinear boundary terms.)

In fact, {\em interactions} between distinct physical phenomena widely occur in both nature and technology. 
To name a few relevant ones, we just recall---besides thermoelastic and acoustic-structure
interactions---fluid-solid or magnetoelectric interactions. 
The modeling of such interactions leads in a natural way to composite systems of PDE comprising 
components which display different dynamical behaviours.
In a mathematical setting, such evolutionary systems of coupled PDE can be represented 
in a simple form by the following abstract system, which also takes into account the possible influence of control actions:
\begin{equation*}
\begin{cases}
y_t = A_1 y + B_1 u + C_1(y,z)  & \quad y|_{t=0}=y^0\in H_1
\\
z_t = A_2 z + B_2 v + C_2(z,y) & \quad z|_{t=0}=z^0\in H_2
\end{cases},
\end{equation*}
where $H_1$ and $H_2$ are appropriate Hilbert spaces, $A_1$ and $A_2$ are infinitesimal
generators of $C_0$-semigroups $e^{A_i t}$ on $ H_i$, $i=1,2$, respectively, with e.g.
$e^{A_1 t}$ assumed to be {\em analytic}. 
The control operators $B_i$, $i=1,2$, are {\em unbounded} operators acting from the
control spaces $U_i$ into $H_i$ (or from $U_i$ into $[D(A^*_i)]'$); the coupling between the two different dynamics may occur as well by means of unbounded operators $C_i$. 

It is then natural to seek the proper set of assumptions on the overall dynamics operator in order to fully exploit the diverse features of the components of the system. 
The ultimate goal is to produce a rigorous theory which yields the feedback synthesis of the optimal control by means of a solution to the (nonlinear) operator Riccati equation, 
which must be shown to be well posed. 
The unboundedness of the control operator, while being a natural and prominent feature of the 
problem, provides the main mathematical challenge to achieve a proper definition of the gain operator, which may not have a sufficiently rich (and hence, meaningful) domain. 
On the other hand, the framework of control systems characterized by singular estimates turned out to be too rigid and failing to encompass many `mixed' dynamics that are of fundamental importance 
in the applications. 

A prime illustration of this fact is a benchmark thermoelastic plate model studied in 
\cite{bucci-las-thermo}. 
The distinct control-theoretic properties exhibited by its abstract representation, according to different sets of mechanical/thermal boundary conditions, revealed that the singular estimate 
which pertains to its parabolic component (that is the equation for the temperature distribution in the plate) does not always fully `propagate' to the entire system. 
Under certain boundary conditions, the coupling brings about a more involved behaviour of the (controlled) coupled dynamics. 
More precisely, in contrast with the case of {\em hinged} boundary conditions, when the system is supplemented with {\em clamped} (mechanical) boundary conditions and is subject to Dirichlet thermal control, the coupled dynamics precludes the validity of a singular estimate 
(for the norm of the operator $e^{At}B$, as a bounded operator from the control space $U$ into 
the state space $Y$); see \cite{bucci-las-thermo} and \cite{abl-1}.

It was indeed this initial/boundary value problem for the system of thermoelasticity to provide our original motivation for introducing in \cite{abl-2} a novel class of abstract linear control systems---which includes the one discussed above as a very special case. 
This class is characterized by a suitable decomposition of the operator $e^{At}B=F(t)+G(t)$, 
where the only component $F(t)$ satisfies a singular estimate, while the component $G(t)$ 
exhibits appropriate regularity properties which account for the (predominant) hyperbolic character and the parabolic one; see \cite[Hypothesis~2.2]{abl-2} for full details.
\\
The goal was then to develop a sufficiently general abstract set-up which could 
encompass significant interconnected PDE systems comprising both hyperbolic and parabolic components, and yet lacking the overall `parabolic-like' character disclosed
by a singular estimate. 
Moreover, we aimed at developing a corresponding theory of the quadratic optimal control problem which escaped restrictions on the observation operator in order to achieve 
well-posedness of the Riccati equations, unlike the case of {\em purely hyperbolic} problems.

\smallskip
Thus, a theory of the LQ-problem for the class of systems described above (and mathematically defined by Hypothesis~2.2 in \cite{abl-2}) in the {\it finite time horizon case} has been 
developed in \cite{abl-2}.
The major novel features of this optimal control theory are shortly outlined 
below; the reader is referred to \cite[Theorem~2.2]{abl-2} for a complete 
description of the obtained results.

First of all, in contrast with previous theories the gain operator $B^*P(t)$ is 
{\it only} densely defined in the state space $Y$ 
(it is {\em bounded} on $\cD(A^\epsilon)$, with arbitrarily small $\epsilon$). 
This turns out to be sufficient to obtain well-posed Differential Riccati equations 
\begin{equation*}
\begin{aligned}
& \frac{d}{dt}(P(t)x,z)_Y+ (A^*P(t)x,z)_Y+(P(t)Ax,z))Y-(B^*P(t)x,B^*P(t)z)_U
\label{e:riccati2-eq} \\[1mm]
& \myspace +(R^*Rx,z)_Y=0 \quad \textrm{for any $x,z\in \cD(A)$.}
\end{aligned}
\end{equation*}
Secondly, solvability of the corresponding quadratic optimal control problems
follows without assuming smoothing properties of the observation operator $R$.
More precisely, the weak requirement on $R$---that is the same as condition \eqref{e:keyasR}  
of Hypothesis~\ref{h:ip2}---allows bounded operators which, roughly, just `maintain regularity', 
such as the {\em identity operator}. 
Consequently, natural quadratic functionals such as the ones which involve the integral of the physical energy of the system are allowed for the first time.
 
Model-specific analyses needs to be carried out on diverse PDE problems, in order to establish the regularity of boundary traces that is required in order to apply the introduced abstract theory
(of the LQ-problem on a finite time interval).
This has been achieved, indeed, in the case of boundary control problems for acoustic-structure 
and fluid-structure interactions as well; see \cite{bucci-applicationes} and 
\cite{bucci-las-cvpde,bucci-las-dcds_s}, respectively. 

\medskip
{\bf Present work. } 
The purpose of this paper is to complement the LQ-theory described above by developing a 
complete {\it infinite time horizon} analysis. 
The task is not straightforward owing to a natural mixing of singularities occurring
in short and long time. 
The interplay between the long time stability for the forward problem and the short time 
development of singularities for the adjoint problem lie at the heart of the problem. 

We emphasize at the outset that we establish solvability of the optimal control problem, 
as well as well-posedness of the corresponding Algebraic Riccati equations,
under {\em minimal} assumptions on the operators involved.
More precisely, we work on the very same abstract framework set forth in the finite time horizon case, that is the one defined by \cite[Hypothesis~2.2]{abl-2} for the operator $e^{At}B$ 
and the observation operator $R$ which occurs in the cost functional \eqref{e:funct}.
The finite time horizon scenario is only complemented with two natural stability requirements pertaining to the underlying semigroup $e^{At}$ and the component $F(t)$ arising from the decomposition of $e^{At}B$; see Hypothesis~\ref{h:ip1} and Hypothesis~\ref{h:ip2}(i) in the next Section.
 
As a general strategy for the proof, we start with the variational approach which has been 
pursued in the work of the authors of \cite{las-trig-encyclopedia}, since their earlier
studies of parabolic PDEs with boundary control; see, e.g., \cite{las-trig-analytic}.
Accordingly, the program evolves through the following steps.
(a) The existence of a unique optimal pair $(\hat{u},\hat{y})$ follows by 
convex optimization arguments.
(b) The Lagrange-dual multipliers method yields the optimality condition for the optimal pair.
Then, the optimal pair $(\hat{u},\hat{y})$ is characterized in terms of the data of the problem.
(c)
An operator $P$ is introduced, defined in terms of the optimal state 
$\hat{y}(t;y_0)=\Phi(t)y_0$; this operator is ultimately shown to satisfy the Algebraic 
Riccati Equation (ARE). 

However, although the general philosophy of the aforesaid approach applies as well to the present abstract class of boundary control systems, the regularity specific to the operator 
$e^{At}B$---equivalently, to the corresponding input-to-state map $L$---brings about novel technical challenges % during the process of achieving well-posed Riccati equations.
in order to establish that the ARE is well-posed. 
Delicate issues are encountered when we are 
\begin{enumerate}
\item[(i)]
to give a meaning to the gain operator $B^*P$, for which boundedness on the state space does not hold: hence, it is sought (and in fact established) on a dense subset; even more, when we are  
\item[(ii)]
to pinpoint the regularity of the map $t\mapsto\Phi(t)B$, which plays a central role in 
the study of the differential properties of the optimal state semigroup $\Phi(t)$---that is eventually shown to be differentiable for $t > 0$ on $\cD(A)$.
\end{enumerate}
We note that $\cD(A)$ is {\it not} a natural domain of the strongly perturbed evolution 
$\Phi(t)$ and therefore the aforementioned differentiability is far from expected. 
This differential property is a consequence of an appropriate (time and space) regularity
result obtained for the map $t\mapsto \Phi(t)B$, which in turn hinges on the a suitably
developed operator perturbation theory applied to the original operator 
$e^{At}B: U \rightarrow [\cD(A^*)]'$.  

It is important to emphasize that while the former issue (i) has its counterpart in the 
finite time horizon theory for the present class of dynamics (see \cite{abl-2}), 
(ii) requires novel developments not only with respect to previous theories of the LQ-problem 
(namely, for different classes of boundary control systems, such as the one characterized by 
singular estimates), but also with respect to the finite time horizon case for the present class of dynamics.
For our developments we introduce and employ an appropriate class of weighted function spaces 
(see \eqref{e:weighted-spaces}), whose central role will become fully apparent in 
Section~\ref{ss:differentiate-Phi(t)B}. 
 
\medskip
{\bf Paper outline. }
We conclude the Introduction with a brief overview of the paper.
In Section~\ref{s:one} we introduce the class of linear control systems, characterized by Assumptions~\ref{h:ip1}-\ref{h:ip2}, for which we develop the infinite time horizon optimal
control theory.
The major statements of this theory are collected in Theorem~\ref{t:main}.
Few remarks about the notation are given at the end of the Section.

In Section~\ref{s:two} we derive a complex of regularity results which concern the operator
$B^*e^{A^*t}$, as well as the components $F$ and $G$ of its decomposition.
In addition, we develop a full regularity theory for the input-to-state map $L$ 
defined by \eqref{e:input-to-state} and its adjoint $L^*$ (Proposition~\ref{p:stimeL}, 
Proposition~\ref{p:stimeL*}).
These results constitute the fundamental basis for the developments of the subsequent 
sections.

In Section~\ref{s:gain} we introduce the optimal cost operator $P$ and show that the gain 
operator $B^*P$ is bounded on $\cD(A^\epsilon)$ (Theorem~\ref{t:bounded-gain}).
The validity of the feedback representation of the optimal control is established here.

In Section~\ref{s:are} we prove well-posedness of the algebraic Riccati equations.
The technical basis for the corresponding Theorem~\ref{t:are-wellposed} is
found in Proposition~\ref{p:heritage} and its Corollary~\ref{c:differentiability-on-domain_A},
which in turn follow from the novel result of Theorem~\ref{t:invertible-on-X_q}.
The preliminary description of the domain $\cD(A_P)$ of the optimal state semigroup, 
along with the detailed information provided by Proposition~\ref{p:berliner-inclusion}
(whose proof relies on the theory of interpolation spaces), will also prove essential 
in showing Corollary~\ref{c:differentiability-on-domain_A}.

In Section~\ref{s:examples} we briefly illustrate the applicability of the obtained 
infinite-time horizon theory through the analysis of a natural optimal control problem for the thermoelastic system which motivated and initiated our former theory of \cite{abl-2}.

%----------------------------------------------------------------------------
\section{The mathematical problem. Main results}   \label{s:one}
%----------------------------------------------------------------------------
Let $Y$ and $U$ be two separable Hilbert spaces, the {\em state} and {\em control} space, respectively.
We consider, on the extrapolation space $[\cD(A^*)]'$, the abstract control system 
\begin{equation}\label{e:state-eq}
\begin{cases}
y'(t)=Ay(t)+Bu(t)\,, & \quad t>0
\\[2mm]
y(0)=y_0\in Y\,, & 
\end{cases}
\end{equation}
under the following basic Assumptions.
\begin{assumption} \label{h:ip1} 
Let $Y$, $U$ separable complex Hilbert spaces. 
\begin{itemize}
\item 
The closed linear operator $A:\cD(A)\subseteq Y \to Y$ is the infinitesimal generator of a strongly continuous semigroup $\{e^{At}\}_{t\ge 0}$ on $Y$, which is exponentially stable;
namely, there exist constants $M\ge 1$ and $\omega>0$ such that 
\begin{equation} \label{e:stability}
\|e^{At}\|_{\cL(Y)} \le M \,e^{-\omega t} \qquad \forall t\ge 0\,;
\end{equation}
then $A^{-1}\in \cL(Y)$;
\item 
$B:U\to [\cD(A^*)]'$ is a bounded linear operator; equivalently, $A^{-1}B\in \cL(U,Y)$.
\end{itemize}
\end{assumption} 

To the state equation \eqref{e:state-eq} we associate the quadratic functional
\begin{equation} \label{e:funct}
J(u)=\int_0^\infty \left(\|Ry(t)\|_Z^2 + \|u(t)\|_U^2\right)dt\,, 
\end{equation}
where $Z$ is a third separable Hilbert space---the so called observation space (possibly
$Z\equiv Y$)---and at the outset the {\em observation} operator $R$ simply satisfies 
\begin{equation}\label{e:basic-for-r}
R\in \cL(Y,Z)\,. % is a bounded linear operator.
\end{equation}

\begin{remarks}
\begin{rm}
(i) Since by Assumption~\ref{h:ip1} the semigroup $e^{At}$ is uniformly stable, 
then $-A$ is a positive operator and the fractional powers $(-A)^\alpha$, $\alpha\in (0,1)$,
are well defined.
In order to make the notation lighter, we shall write $A^\alpha$ instead of $(-A)^\alpha$ throughout the paper. 
\\
(ii) We note that the functional \eqref{e:funct} makes sense at least for $u\equiv 0$.
This again in view of the exponential stability of the semigroup $e^{At}$ (Assumption~\ref{h:ip1}), 
which combined with \eqref{e:basic-for-r} ensures $Ry(\cdot,y_0;0)\in L^2(0,\infty;Y)$.
\\
(iii) The analysis carried out in the present paper easily extends to more general quadratic functionals, like
\begin{equation*}
J(u)=\int_0^\infty \big(\|Ry(t)\|_Z^2 + \|\tilde{R}u(t)\|_U^2\big)dt\,,
\end{equation*}
provided $\tilde{R}$ is a coercive operator in $U$. 
We take $\tilde{R}=I$ just for the sake of simplicity and yet without loss of generality.
\end{rm}
\end{remarks}

The optimal control problem under study is formulated in the usual way.

\begin{problem}[The optimal control problem] \label{p:problem-0}
Given $y_0\in Y$, seek a control function $u\in L^2(0,T;U)$ which minimizes the
cost functional \eqref{e:funct}, where $y(\cdot)=y(\cdot\,;y_0,u)$ is the solution to
\eqref{e:state-eq} corresponding to the control function $u(\cdot)$.
\end{problem}

Aimed to pursue the study of the infinite horizon problem for the abstract class of boundary 
control systems first introduced in \cite{abl-2}, we assume throughout the paper that 
the dynamics, control and observation operators are subject to the following 
conditions.

\begin{assumptions}\label{h:ip2} 
The operator $B^*e^{A^*t}$ can be decomposed as 
\beq \label{e:key-hypo} 
B^*e^{A^*t}x = F(t)x + G(t)x, \qquad t\ge 0, \quad x\in \cD(A^*)\,,
\eeq
where $F(t):Y\to U$ and $G(t):\cD(A^*)\to U$, $t>0$, are bounded linear
operators satisfying the following assumptions:
\begin{description}
\item[(i)] 
there exist constants $\gamma\in (0,1)$ and $\eta, N>0$ such that 
\beq \label{e:keyasF}
\|F(t)\|_{\cL(Y,U)} \le N\,t^{-\gamma}\,e^{-\eta t}\qquad \forall t>0\,;
\eeq
\item[(ii)] 
there is a time $T>0$ such that the operator $G(\cdot)$ belongs 
to $\cL(Y,L^p(0,T;U))$ for all $p\in [1,\infty[$;
\item[(iii)] 
there exist $T>0$ and $\epsilon>0$ such that:
\begin{description}
\item[(a)] 
the operator $G(\cdot){A^*}^{-\epsilon}$ belongs to $\cL(Y,C([0,T];U))$, 
and in particular
\beq \label{e:keyasGb}
K_T:= \sup_{t\in [0,T]}\|G(t){A^*}^{-\epsilon}\|_{\cL(Y,U)} <\infty\,;
\eeq
\item[(b)] 
the operator $R^*R$ belongs to $\cL(\cD(A^{\epsilon}),\cD({A^*}^{\epsilon}))$, i.e.
\beq \label{e:keyasR} 
\|{A^*}^{\epsilon}R^*RA^{-\epsilon}\|_{\cL(Y)} \le c<\infty\,;
\eeq
\item[(c)] 
there exists $q\in\, (1,2)$ such that the operator 
$B^*e^{A^*\cdot }{A^*}^\epsilon$, which is defined in $\cD({A^*}^{1+\epsilon})$, has an 
extension which belongs to $\cL(Y,L^q(0,T;U))$. 
\end{description}
\end{description}
\end{assumptions}

% Enunciati dei teoremi fondamentali

\subsection{Statement of the main results}

The main result of this paper is the following Theorem, which collects the most significant
statements which pertain to the solution of the optimal control problem, as well as
to well-posedness of the corresponding algebraic Riccati equations. 

\begin{theorem} \label{t:main}
Consider the optimal control Problem~\ref{p:problem-0}, under the
set of Hypotheses~\ref{h:ip1}-\ref{h:ip2}.
Then, the following statements are valid.
\begin{enumerate}

\item[S1.] 
For any $y_0\in Y$ there exists a unique optimal pair 
$(\hat{u}(\cdot),\hat{y}(\cdot))$ for Problem~\ref{p:problem-0},
which satisfies the following regularity properties
\begin{equation}\label{e:regularity-optimalpair}
\hat{u}\in \bigcup_{2\le p<\infty} L^p(0,\infty;U)\,,
\quad 
\hat{y}\in C_b([0,\infty);Y) \cap \big[\bigcup_{2\le p< \infty} L^p(0,\infty;Y)\big]\,.
\end{equation}  

\item[S2.] 
The operator $\Phi(t)$, $t\ge 0$, defined by 
\begin{equation}\label{e:optimal-state-semigroup}
\Phi(t)y_0:=\hat{y}(t)=y(t,y_0;\hat{u})
\end{equation}
is a strongly continuous semigroup on $Y$, $t\ge 0$, which is exponentially stable;
namely, 
\begin{equation} \label{e:exponential_1}
\exists M_1\ge 1\,, \,\omega_1>0: \qquad \|\Phi(t)\|_{\cL(Y)} \le M_1 \,e^{-\omega_1t} \quad \forall t\ge 0\,.
\end{equation}

\item[S3.] 
The operator $P\in\cL(Y)$ defined by 
\begin{equation} \label{e:optimal-cost-op}
Px:= \int_0^\infty e^{A^*t}R^*R \Phi(t)x\, dt \qquad x\in Y
\end{equation}
is the optimal cost operator: namely, % the optimal cost is the quadratic form 
\begin{equation*}
(Px,x)_Y=\int_0^\infty \big(\|R\hat{y}(t;x))\|_Z^2 + \|\hat{u}(t;x)\|_U^2\big)dt\,,
\qquad \forall x\in Y\,,
\end{equation*}
which also shows that $P$ is (self-adjoint and) non-negative.

\item[S4.] 
The gain operator $B^*P$ belongs to $\cL(\cD(A^\epsilon),U)$;
namely, it is just densely defined on $Y$ and yet it is bounded on $\cD(A^\epsilon)$.

\item[S5.] 
The infinitesimal generator $A_P$ of the (optimal state) semigroup $\Phi(t)$ 
coincides with the operator $A(I-A^{-1}BB^*P)$, on the domain
\begin{align*}
\cD(A_P)& \subset \big\{x\in Y: x-A^{-1}BB^*Px\in \cD(A) \big\}
\\
& \subset \big\{ x\in Y: \exists \wlim_{t\to 0^+} \frac1{t} \int_0^t e^{A(t-\tau)}A^{-1}BB^*P\Phi(\tau)x\,d\tau 
%= \\
%& \myspace\myspace\myspace = A^{-1}BB^*Px\ \textrm{in } Y\big\}.
\end{align*}

\item[S6.] 
The operator $e^{At}B$, defined in $U$ and a priori with values in $[\cD(A^*)]'$,
is such that
\begin{equation}
e^{\delta\cdot}e^{A\cdot}B\in \cL(U,L^p(0,\infty;[\cD({A^*}^{\epsilon})]')) 
\qquad \forall p\in [1,1/{\gamma})
\end{equation}
for all $\delta\in [0,\omega\wedge \eta)$;
almost the very same regularity is inherited by the operator $\Phi(t)B$:
\begin{equation}
e^{\delta\cdot}\Phi(\cdot)B\in \cL(U,L^p(0,\infty;[\cD({A^*}^{\epsilon})]')) 
\qquad \forall p\in [1,1/{\gamma})\,,
\end{equation}
provided $\delta\in [0,\omega\wedge \eta)$ is sufficiently small.

\item[S7.] 
The optimal cost operator $P$ defined in \eqref{e:optimal-cost-op} 
satisfies the following additional regularity properties:
\begin{equation*}
P\in \cL(\cD(A_P),\cD(A^*))\cap \cL(\cD(A),\cD(A^*_P))\,; 
\end{equation*}
moreover, $P$ is a solution to the Algebraic Riccati equation corresponding to 
Problem~\ref{p:problem-0}, that is 
\begin{equation*}
\begin{split}
& (Px,Az)_Y+(Ax,Pz)_Y-(B^*Px,B^*Pz)_U+(Rx,Rz)_Z=0 
%\label{e:riccati-eq} 
\\[1mm]
& \myspace\myspace\myspace\textrm{for any $x,z\in \cD(A)$,}
\end{split}
\end{equation*}
to be interpreted as   
\begin{equation*}
\begin{split}
& (A^*Px,z)_Y+(x,A^*Pz)_Y-(B^*Px,B^*Pz)_U+(Rx,Rz)_Z=0 
%\label{e:riccati-eq} 
\\[1mm]
& \myspace\myspace\myspace\textrm{when $x,z\in \cD(A_P)$.}
\end{split}
\end{equation*}

\item[S8.] 
If $x\in\cD(A^\epsilon)$, the regularity \eqref{e:regularity-optimalpair} of the optimal pair 
is improved as follows:
\begin{equation*}
\begin{split}
& \hat{y}\in C_b([0,\infty);\cD(A^\epsilon)) 
\cap \big[\bigcup_{1\le p\le \infty} L^p(0,\infty;\cD(A^\epsilon))\big]\,,
\\ 
& \hat{u}\in C_b([0,\infty);U)\cap \big[\bigcup_{1\le p\le \infty} L^p(0,\infty;U)\big]\,.
\end{split}
\end{equation*}  

\item[S9.] 
The following (pointwise in time) feedback representation of the optimal control is valid
for any initial state $x\in Y$:
\begin{equation*}
\hat{u}(t,x) = - B^*P \hat{y}(t,x) \qquad \textrm{for a.e. $t\in (0,\infty)$.}
\end{equation*}

\end{enumerate}

\end{theorem}

%----------------------------------------------------------------------------------------------
\subsection{Notation} 
%---------------------------------------------------------------------------------------------
Inner products in Hilbert spaces $X$ ($Y$ and $U$, in practice) will be denoted by
$(\cdot,\cdot)_X$ throughout the paper.
The subscripts will be omitted when no confusion arises.
Instead, the symbol $\langle\cdot,\cdot\rangle_V$ will denote a duality pairing of 
$V'$ with $V$; $V=\cD({A^*}^\epsilon)$ will occur most often.

We shall utilize $L^p$-spaces with weights, defined (in the usual way) as follows:
\begin{equation*}
L^p_g(0,\infty;X):=\Big\{ f:(0,\infty)\to X\,, \; g(\cdot)\,f(\cdot)\in L^p(0,\infty;X)\Big\}\,,
\end{equation*}
where $g:(0,\infty)\to \mathbb{R}$ is a given (weight) function.
We will more specifically utilize {\em exponential weights} such as $g(t) =e^{\delta t}$;
to simplify the notation we will write
\begin{equation} \label{e:weighted-spaces}
L^p_\delta(0,\infty;X):=\Big\{ f:(0,\infty)\to X\,, % \;\textrm{$\|f\|_X$ is Lebesgue measurable, 
\; e^{\delta \cdot}\,f(\cdot)\in L^p(0,\infty;X)\Big\}\,.
\end{equation}

%----------------------------------------------------------------------------------------------
\section{The input-to-state map: relevant regularity results} \label{s:two}
%----------------------------------------------------------------------------------------------
We begin by providing a series of regularity results, concerning first the operator $B^*e^{A^*t}$ 
(or one of its components) and then the input-to-state map $L$ defined by \eqref{e:input-to-state}
(and its adjoint $L^*$).
These results constitute the first consequences of the abstract Assumptions~\ref{h:ip2} as well
as the fundamental basis for the more challenging developments of the subsequent sections.

\subsection{Preliminary results}

\begin{proposition} \label{p:expG} 
For each $\delta\in [0,\omega\wedge\eta[$ and $p\in [1,\infty[$ the map $t \mapsto e^{\delta t} G(t)$ belongs to $\cL(Y,L^p(0,\infty;U))$.
\end{proposition}

\begin{proof}
Let $T>0$ be such that Hypotheses~\ref{h:ip2}(ii) holds. 
By hypothesis~\ref{h:ip2}(i), we have
\begin{equation*}
F(\cdot) \in \cL(Y,L^p(\epsilon,T/2;U)) \quad \forall \epsilon \in ]0,\frac{T}{2}]\,.
\end{equation*}
Hence,
\beq \label{stima1}
\int_{T/2}^T \|B^*e^{A^*t}x\|_U^p\, dt \le c\|x\|_Y^p\quad \forall x\in \cD(A^*)\,.
\eeq
Now we can write:
\beyy
\int_{T/2}^\infty \|e^{\delta t}B^*e^{tA^*}x\|_U^p\, dt 
& = & \int_{T/2}^\infty e^{\delta pt}\|B^*e^{A^*t}x\|_U^p\, dt = 
\\ 
& = & \sum_{k=1}^\infty \int_{kT/2}^{(k+1)T/2}e^{\delta pt}\|B^*e^{A^*t}x\|_U^p\, dt = 
\\
& = & \sum_{k=1}^\infty \int_{T/2}^Te^{\delta ps}e^{\delta p(k-1)\frac{T}2}
\|B^*e^{A^*s}[e^{A^*(k-1)\frac{T}2 }x]\|_U^p\, ds\,;
\eeyy
using (\ref{stima1}), we deduce for each $x\in \cD(A^*)$
\beyy
\int_{T/2}^\infty \|e^{\delta t}B^*e^{A^*t}x\|_U^p\, dt & \le & 
c\,e^{\delta pT}\sum_{k=1}^\infty e^{\delta p(k-1)\frac{T}2} \|e^{A^*(k-1)\frac{T}2}x\|_Y^p\, 
\le \\
& \le & c\,e^{\delta pT}M^p \sum_{k=1}^\infty e^{-(\omega-\delta)p\frac{T}2(k-1)} \|x\|_Y^p 
= c(p,T)\|x\|_Y^p\,.
\eeyy
By density, this shows that the map $t\mapsto e^{\delta t}B^*e^{A^*t}$ has a continuous extension from $Y$ to $L^p(T/2,\infty;U)$. 
On the other hand, Hypothesis~\ref{h:ip2}(i) implies that 
\begin{equation*}
\int_{T/2}^\infty \|e^{\delta t}F(t)x\|_U^p\, dt 
\le N^p \left(\frac2{T}\right)^{\gamma p} \int_{T/2}^\infty e^{-(\eta-\delta)pt} \|x\|_Y^p\,dt 
\le c(p,T)\,\|x\|_Y^p\,,
\end{equation*}
so that
\begin{equation*}
\int_{T/2}^\infty \|e^{\delta t}G(t)x\|_U^p\, dt = 
\int_{T/2}^\infty \|e^{\delta t}\left[B^*e^{A^*t}-F(t)\right]x\|_U^p\, dt \le c(p,T)\|x\|_Y^p\,.
\end{equation*}

As, by Hypothesis~\ref{h:ip2}(ii),
\begin{equation*}
\int_0^{T/2}\|e^{\delta t}G(t)x\|_U^p\, dt \le e^{\delta p \frac{T}2}\int_0^{T/2}\|G(t)x\|_U^p\, dt 
\le c(p,T)\|x\|_Y^p\,,
\end{equation*}
we conclude that $t\mapsto e^{\delta t}G(t)$ is continuous from $Y$ to $L^p(0,\infty;U)$. 
\end{proof}

\begin{proposition}\label{p:expBA} 
For each $\delta\in [0,\omega\wedge \eta[$ the map 
$t \mapsto e^{\delta t}B^*e^{A^*t}{A^*}^\epsilon$ has an extension which belongs to $\cL(Y,L^q(0,\infty;U))$. 
\end{proposition}

\begin{proof}
Let $T>0$ such that Hypothesis~\ref{h:ip2}(iii)(c) holds. 
We can write, for each $x\in \cD({A^*}^\epsilon)$,
\beyy
\int_0^\infty \|e^{\delta t}B^*e^{A^*t}{A^*}^\epsilon x\|_U^q\,dt 
& = & \sum_{k=0}^\infty \int_{kT}^{(k+1)T}e^{\delta qt}\|B^*e^{A^*t}{A^*}^\epsilon x\|_U^q\,dt =
\\
& = & \sum_{k=0}^\infty \int_0^T e^{\delta q(s+kT)}\|B^*e^{A^*s}{A^*}^\epsilon e^{A^*kT}x\|_U^q\,ds\,.
\eeyy
As $e^{A^*kT}x\in \cD({A^*}^\epsilon)$ whenever $x\in \cD({A^*}^\epsilon)$, we get by 
Hypothesis~\ref{h:ip2}(iii)(c)
\beyy
\int_0^\infty \|e^{\delta t}B^*e^{A^*t}{A^*}^\epsilon x\|_U^q\,ds & \le & e^{\delta qT}\sum_{k=0}^\infty e^{\delta qkT}\int_0^T \big\|B^*e^{A^*s}{A^*}^\epsilon\big[e^{A^*kT}x\big]\big\|_U^q \,ds 
\le \\
& \le & c\,e^{\delta qT}\sum_{k=0}^\infty M^q e^{-(\omega-\delta)qkT}\|x\|_Y^q\,.
\eeyy
\end{proof}

\begin{proposition}\label{p:expGf} 
For each $\delta\in [0,\omega\wedge\eta[$ and $p\in [1,1/\gamma[$ the map 
$t \mapsto e^{\delta t} G(t)f(t)$ belongs to $L^p(0,\infty;U)$ whenever 
$f\in L^r(0,\infty;\cD({A^*}^\epsilon))$, with
\begin{equation} \label{e:constraint-r}
\frac{p}{1-\gamma p}<r\le \infty\,;
\end{equation}
in addition,
\begin{equation}\label{e:expGf}
\|e^{\delta\cdot}G(\cdot)f(\cdot)\|_{L^p(0,\infty;U)} 
\le c_p\|f\|_{L^r(0,\infty;\cD({A^*}^\epsilon))}\,.
\end{equation}
The constraints assumed on the exponents $p$ and $r$ are sharp.
\end{proposition}

\begin{proof}
Let $f\in L^r(0,\infty;\cD({A^*}^\epsilon))$ be given, initially with
$1\le r<\infty$.
We take a representative of $f(t)$, defined for almost any $t\in [0,T]$,
and derive a preliminary estimate of $\|e^{\delta\cdot}G(\cdot)f(\cdot)\|$ in 
$L^p(0,T;U)$, that is 
\begin{equation}\label{e:finite_1}
\begin{split}
\int_0^T \|e^{\delta t}\,G(\cdot)f(\cdot)\|_U^p\,dt 
&\le K_T^p \int_0^T e^{\delta pt}\|{A^*}^\epsilon f(t)\|_Y^p\, dt
\\[1mm]
& \le K_T^p \Big(\int_0^T e^{\delta t\,pr/(r-p)}\,dt \Big)^{(r-p)/r}\, 
\|f\|_{L^r(0,\infty;\cD({A^*}^\epsilon))}^p
\\[1mm]
& = C_T\|f\|_{L^r(0,\infty;\cD({A^*}^\epsilon))}^p
\end{split}
\end{equation} 
where we utilized first Hypothesis~\ref{h:ip2}(iii)(a) and then we applied the H\"older inequality
with exponents $r/(r-p)$ and $r/p$, taking $r>p\ge 1$. 

On the other hand, owing to Hypothesis~\ref{h:ip2}(i), we immediately find the estimate
\begin{equation}\label{e:pre-finite_2}
\int_0^T e^{\delta pt}\|F(t)f(t)\|_U^p \, dt 
\le N^p\int_0^T\,\frac{e^{-(\eta-\delta)pt}}{t^{\gamma p}}\,\|f(t)\|_Y^p\,.
\end{equation}
Thus, in order to render the integrand in the right hand side of \eqref{e:pre-finite_2} 
summable, we take $p\in [1,1/\gamma)$ first, and apply once more the H\"older inequality, this time with exponents $s/p$ and $s/(s-p)$, with $s= pr/(r-p)$, thus obtaining
\begin{equation}\label{e:finite_2}
\begin{split}
\int_0^T e^{\delta pt}\|F(t)f(t)\|_U^p \, dt 
& \le N^p\Big(\int_0^\infty \,\frac{e^{-(\eta-\delta)st}}{t^{\gamma s}}\,dt\Big)^{p/s}\,
\|f\|_{L^r(0,\infty;\cD({A^*}^\epsilon))}^p
\\[1mm]
& = C_p\|f\|_{L^r(0,\infty;\cD({A^*}^\epsilon))}^p\,,
\end{split}
\end{equation}
We note that in \eqref{e:finite_2} we used $s/p>1$, along with $s\gamma=pr\gamma/(r-p)<1$,
which readily yields the lower bound in \eqref{e:constraint-r}.
Combining \eqref{e:finite_2} with \eqref{e:finite_1} we find 
\begin{equation} \label{e:estimate-on-bddinterval}
\begin{split}
\int_0^T e^{\delta pt}\|B^*e^{A^*t}f(t)\|_U^p\,dt & \le  
2^p \int_0^T e^{\delta pt}\big( \|F(t)f(t)\|_U^p + \|G(t)f(t)\|_U^p\big)\,dt
\\[1mm]
& \le  C(p,T) \|f\|^p_{L^r(0,\infty;\cD({A^*}^\epsilon))}\,.
\end{split}
\end{equation}

\smallskip
The obtained estimate \eqref{e:estimate-on-bddinterval} pertaining to the integral on $(0,T)$
is now used to derive the following one on $(T,\infty)$:  
\begin{eqnarray}
\lefteqn{\hskip -5mm
\int_T^\infty e^{\delta pt}\|B^*e^{A^*t}f(t)\|_U^p\,dt 
= \sum_{k=1}^\infty \int_{kT}^{(k+1)T}e^{\delta pt}\|B^*e^{A^*t}f(t)\|_U^p\,dt =}
\nonumber \\
& & = \sum_{k=1}^\infty e^{\delta pkT} \int_0^T e^{\delta p\tau}
\|B^*e^{A^*\tau}e^{A^*kT}f(\tau+kT)\|_U^p\,d\tau 
\nonumber\\
& & \le c(p,T)\sum_{k=1}^\infty e^{\delta pkT}
\|e^{A^*kT}f(\cdot+kT)\|^p_{L^r(0,\infty;\cD({A^*}^\epsilon))}
\nonumber\\
& & \le c(p,T)M^p \sum_{k=1}^\infty e^{-(\omega-\delta) pkT}
\|f\|^p_{L^r(0,\infty;\cD({A^*}^\epsilon))}\,.
\label{e:convergent-series}
\end{eqnarray}
Since $\delta<\omega$, the series in \eqref{e:convergent-series} is convergent, and there exists a constant $C_p$, which depends only on $p$, such that 
\begin{equation}\label{e:tobeused-aswell}
\int_T^\infty e^{\delta pt}\|B^*e^{A^*t}f(t)\|_U^p\,dt 
\le C_p\|f\|^p_{L^r(0,\infty;\cD({A^*}^\epsilon))}\,.
\end{equation}

Assume now on $p$ and $r$ the constraints arisen so far during the proof, namely
\begin{equation*}
1\le p<\frac{1}{\gamma}\,, \qquad \frac{p}{1-\gamma p}<r<\infty\,.
\end{equation*}
The same arguments used to find \eqref{e:finite_2} provide
\begin{equation*}
\begin{split}
\int_T^\infty e^{\delta pt}\|F(t)f(t)\|_U^p\,dt 
& \le M^p \int_T^\infty t^{-\gamma p} e^{-(\eta-\delta)pt}\|f(t)\|^p_Y\,dt  
\\
& \le c(p)\|f\|^p_{L^r(0,\infty;\cD({A^*}^\epsilon))}\,,
\end{split}
\end{equation*}
which in view of \eqref{e:tobeused-aswell} implies as well 
\begin{align*}
\int_T^\infty e^{\delta pt}\|G(t)f(t)\|_U^p\,dt 
& = \int_0^\infty e^{\delta pt}\|[F(t)-B^*e^{A^*t}]f(t)\|_U^p\,dt 
\\
& \le c(p,T)\|f\|^p_{L^r(0,\infty;\cD({A^*}^\epsilon))}\,.
\end{align*}
The above estimate and \eqref{e:finite_1} finally establish \eqref{e:expGf}.

\smallskip
The proof of \eqref{e:expGf} in the case $r=+\infty$ is similar (even simpler), hence it is omitted.
\end{proof}

%------------------------------------------------------------------
\subsection{Regularity of the input-to-state map and its adjoint} 
%------------------------------------------------------------------
A fundamental prerequisite for all the computations performed in the paper is a detailed 
description of the regularity properties of the `input-to-state' mapping, that is the mapping 
$L$ which associates to any control function $u(\cdot)$ the solution to the state equation 
\eqref{e:state-eq} with $y_0=0$.
Namely, $L$ is defined as follows:
\begin{equation} \label{e:input-to-state}
\begin{aligned}
L:v\mapsto Lv\,, \quad Lv(t) & := \int_0^t e^{(t-s)A}Bv(s)\,ds 
\\ &\, = \int_0^t F(t-s)^*v(s)\,ds + \int_0^t G(t-s)^*v(s)\,ds 
\\ 
& \,= L_{(1)}v(t)+L_{(2)}v(t), \qquad t\ge 0\,.
\end{aligned}
\end{equation}

\smallskip
\noindent
We begin by recalling the main regularity properties of the integral operator 
$L_{(1)}$. We note that in view of the key estimate \eqref{e:keyasF} satisfied by 
the component $F(t)$, the proof of the statements of Proposition~\ref{p:stimeL1} 
below follows a pretty standard route. 
In fact, it employs the same arguments used in the study of the input-to-state map pertaining to parabolic-like dynamics or systems yielding singular estimates; see \cite{las-trig-encyclopedia} and \cite{las-trig-se-1}.
 
\begin{proposition}\label{p:stimeL1} 
Let $L_{(1)}$ be the operator defined in \eqref{e:input-to-state}. 
Then, the following regularity properties hold.
\begin{description}
\item[(i)] $L_{(1)}$ maps continuously $L^1(0,\infty;U)$ into $L^r(0,\infty;Y)$ 
for each $r\in [1,\frac1{\gamma}[$;
\item[(ii)] $L_{(1)}$ maps continuously $L^p(0,\infty;U)$ into $L^r(0,\infty;Y)$ for each 
$p\in ]1,\frac1{1-\gamma}[$ and $r\in [p,\frac{p}{1-(1-\gamma)p}]$;
\item[(iii)] $L_{(1)}$ maps continuously $L^{\frac{1}{1-\gamma}}(0,\infty;U)$ into 
$L^r(0,\infty;Y)$ for each $r\in [\frac1{1-\gamma},\infty[$;
\item[(iv)] $L_{(1)}$ maps continuously $L^p(0,\infty;U)$ into 
$L^r(0,\infty;Y)\cap C_b([0,\infty[;Y)$ for each $p\in ]\frac1{1-\gamma}, \infty]$ and $r\in [p,\infty]$.
\end{description}
\end{proposition}

\begin{proof}
First, an easy application of H\"older inequality and Tonelli Theorem shows that $L_{(1)}$ maps continuously $L^p(0,\infty;U)$ into $L^p(0,\infty;Y)$ for each $p\in [1,\infty]$. Next, property (i) follows directly by H\"older inequality; property (ii) is a consequence of \cite[Theorem~383]{hlp} and interpolation, and finally properties (iii) and (iv) follow again by 
the H\"older inequality and interpolation. 
\end{proof}

\smallskip
The analysis of the operator $L_{(2)}$ is more tricky.
It exploits the distinct regularity properties of the component $G(t)$ pointed out
in Proposition~\ref{p:expGf}.
\begin{proposition}\label{p:stimeL2} 
The following properties hold true:
\begin{description}
\item[(i)] 
$L_{(2)}$ maps continuously $L^1(0,\infty;U)$ into $L^s(0,\infty;[\cD({A^*}^\epsilon)]')$ 
for each $s\in[1,1/\gamma)$;
\item[(ii)] 
$L_{(2)}$ maps continuously $L^p(0,\infty;U)$ into $L^r(0,\infty;Y)\cap C_b([0,\infty[;Y)$ for each $p\in (1,\infty]$ and $r\in [p,\infty]$. 
\end{description}
\end{proposition}

\begin{proof}
{\bf (i)} 
The proof is based on a duality argument.
Let $u\in L^1(0,\infty,U)$ be given and let $w\in L^{r}(0,\infty;\cD({A^*}^\epsilon))$,
with the summability exponent $r$ to be choosen appropriately later.
We take the duality pairing $\langle (L_{(2)}u)(t),w(t)\rangle_{\cD({A^*}^\epsilon)}$,
integrate on $(0,\infty)$ and rewrite as follows: 
\begin{equation*}
\begin{aligned}
\int_0^\infty \langle L_{(2)}u(t),w(t)\rangle_{\cD({A^*}^\epsilon)}\,dt 
& =\int_0^\infty \big\langle 
\int_0^t G(t-\tau)^*u(\tau)\,d\tau,w(t)\big\rangle_{\cD({A^*}^\epsilon)}\,dt 
\\
& = \int_0^\infty \int_\tau^\infty \big(u(\tau), G(t-\tau)w(t)\big)_U\,dt\,d\tau
\\
& = \int_0^\infty \big(u(\tau),\int_\tau^\infty G(\sigma)w(\sigma+\tau)\,d\sigma \big)_U
\,d\tau\,,
\end{aligned}
\end{equation*}
which gives 
\begin{equation} \label{e:preliminary-estimate}
\Big|\int_0^\infty \langle L_{(2)}u(t),w(t)\rangle_{\cD({A^*}^\epsilon)}\,dt\Big|
\le \|u\|_{L^1(0,\infty;U)}
\,\sup_{\tau\ge 0}\Big[\int_\tau^\infty \|G(\sigma)w(\sigma+\tau)\|_U\, d\sigma\Big]\,.
\end{equation}
We then proceed with the estimate of the integral on the right hand side of 
\eqref{e:preliminary-estimate}.
We insert the exponential weight $e^{\delta t}$, apply the H\"older inequality first ($p'$ denotes the conjugate exponent of $p$) and utilize the estimate \eqref{e:expGf} of Proposition~\ref{p:expGf} next, to find   
\begin{equation}\label{e:dual-estimate}
\begin{aligned}
& \Big|\int_0^\infty \langle L_{(2)}u(t),w(t)\rangle_{\cD({A^*}^\epsilon)}\,dt\Big|
\\
& \quad \le \|u\|_{L^1(0,\infty;U)}\,\Big(\int_0^\infty e^{-\delta p'\sigma}\,d\sigma\Big)^{1/p'}
\, \sup_{\tau\ge 0}\Big[\int_{\tau}^\infty e^{\delta p\sigma}
\|G(\sigma)w(\sigma+\tau)\|^p_U\, d\sigma\Big]^{1/p}
\\
& \quad \le C_T\,\|u\|_{L^1(0,\infty;U)}
\,\sup_{\tau\ge 0}\|w\|_{L^r(\tau,\infty;\cD({A^*}^\epsilon))}
\\
& \quad = C_T\,\|u\|_{L^1(0,\infty;U)}\,\|w\|_{L^r(0,\infty;\cD({A^*}^\epsilon))}\,.
\end{aligned}
\end{equation}
Note carefully that Proposition~\ref{p:expGf} applies with any $p\in [1,1/\gamma)$ and
$w\in L^r(\tau,\infty;\cD({A^*}^\epsilon))$, provided the summability exponent $r$ 
fulfils the constraint \eqref{e:constraint-r}.
Therefore, \eqref{e:dual-estimate} implies by duality 
\begin{equation*}
L_{(2)}\in \cL(L^1(0,\infty;U),L^s(0,\infty;[\cD({A^*}^\epsilon)]'))\,,
\end{equation*}
where $s$ is the conjugate exponent of $r$ and the estimate
\begin{equation*}
\|L_{(2)}u\|_{L^s(0,\infty;[\cD({A^*}^\epsilon)]'))}\le C\,\|u\|_{L^1(0,\infty;U)}\,,
\end{equation*}
holds with a constant $C$ depending only on $p$ and $T$.
Thus, \eqref{e:constraint-r} readily implies 
\begin{equation*}
1\le s< \frac{p}{p-(1-\gamma p)}\,,
\end{equation*}  
while it is elementary to check that 
\begin{equation*}
\max_{p\in [1,\frac1{\gamma})}\,\frac{p}{p-(1-\gamma p)} =\frac 1{\gamma}\,,
\end{equation*}
which confirms the constraint $s\in [1,1/\gamma)$ and thus completes the proof of part (i).

\smallskip
\noindent
{\bf (ii)} For fixed $p\in ]1,\infty[$, using Proposition~\ref{p:expG} and Hypothesis~\ref{h:ip2}(ii), we compute for any $\varphi\in C^\infty_0(]0,\infty[;Y)$:
\beyy
\lefteqn{\left| \int_0^\infty \left(L_{(2)}v(t),\varphi(t)\right)_Y \,dt\right| = \left| \int_0^\infty \int_0^t \left(v(s),G(t-s)\varphi(t)\right)_U \,dsdt \right| \le} \\
& & \le \int_0^\infty \int_0^t \|e^{-\delta(t-\cdot)}v(\cdot)\|_{L^p(0,t;U)} \|e^{\delta \cdot} G(\cdot)\varphi(t)\|_{L^{p'}(0,t;U)}\,dsdt \le \\
& & \le c(p) \int_0^\infty \|e^{-\delta(t-\cdot)}v(\cdot)\|_{L^p(0,t;U)}\|\varphi(t)\|_Y\,dt. 
\eeyy
From here, using H\"older inequality and Tonelli Theorem, it follows easily that 
$$\left| \int_0^\infty \left(L_{(2)}v(t),\varphi(t)\right)_Y \,dt\right| \le \|v\|_{L^p(0,\infty;U)} \|\varphi\|_{L^{p'}(0,\infty;Y)},$$
as well as
$$\left| \int_0^\infty \left(L_{(2)}v(t),\varphi(t)\right)_Y \,dt\right| \le \|v\|_{L^p(0,\infty;U)} \|\varphi\|_{L^1(0,\infty;Y)},$$
which proves, by density, that $L_{(2)}v\in L^p(0,\infty;U)\cap L^\infty(0,\infty;U)$. 
In addition, the argument of \cite[Proposition~B.2]{abl-2} yields continuity. 
The result then follows by interpolation.
\end{proof}

The major regularity properties of $L$ are collected in the following Proposition. 

\begin{proposition}\label{p:stimeL} 
The operator $L$ enjoys the following properties:
\begin{description}
\item[(i)] $L$ maps continuously $L^1(0,\infty;U)$ into $L^r(0,\infty;[\cD({A^*}^\epsilon)]')$ for each $r\in [1,\frac1{\gamma}[$; 
\item[(ii)] $L$ maps continuously $L^p(0,\infty;U)$ into $L^r(0,\infty;Y)$ for each $p\in ]1,\frac1{1-\gamma}[$ and $r\in [p,\frac{p}{1-(1-\gamma)p}]$;
\item[(iii)] \hspace{-1mm}$L$ maps continuously $L^{\frac{1}{1-\gamma}}(0,\infty;U)$ into $L^r(0,\infty;Y)$ for each $r\hspace{-1mm}\in\hspace{-1mm} [\frac1{1-\gamma},\infty[$;
\item[(iv)] $L$ maps continuously $L^p(0,\infty;U)$ into $L^r(0,\infty;Y)\cap C_b([0,\infty[;Y)$ for each $p\in ]\frac1{1-\gamma}, \infty]$ and $r\in [p,\infty]$.
\item[(v)] $L$ maps continuously $L^r(0,\infty;U)$ with $r\in [q',\infty]$ into 
$C_b([0,\infty[;\cD(A^\epsilon))$.
\end{description}
\end{proposition}

\begin{proof}
Parts (i)-(ii)-(iii)-(iv) readily follow combining the results established 
in Propositions~\ref{p:stimeL1} and \ref{p:stimeL2}.

In order to prove (v), let $r\ge q'$ and let $u\in L^r(0,\infty;U)$ be given. 
Then, by Proposition~\ref{p:expBA}, for each $z\in \cD({A^*}^\epsilon)$ and $t\ge 0$ we have:
\beyy
\left| (Lu(t),{A^*}^\epsilon z)_Y \right| & = & \left|\int_0^t (u(t-s),B^*e^{A^*s}{A^*}^\epsilon z)_U\,ds \right| = \\
& = & \left|\int_0^t (e^{-\delta s}u(t-s),e^{\delta s}B^*e^{A^*s}{A^*}^\epsilon z)_U\,ds \right| \le \\
& \le & c \|e^{-\delta \cdot} u(t-\cdot)\|_{L^{q'}(0,t;U)}\| e^{\delta \cdot}B^*e^{A^*\cdot}{A^*}^\epsilon z\|_{L^q(0,t;U)} \le \\
& \le & c(q,r) \|u\|_{L^r(0,\infty;U)}\|z\|_Y\,,
\eeyy
so that $Lu\in L^\infty(0,\infty;\cD(A^\epsilon))$. Now if $t>\tau\ge 0$ we also have
\beyy
\lefteqn{\left| (Lu(t)-Lu(\tau),{A^*}^\epsilon z)_Y \right|=}\\
& & =\left|\int_0^t (u(t-s),B^*e^{A^*s}{A^*}^\epsilon z)_U\,ds - \int_0^\tau (u(\tau-s),B^*e^{A^*s}{A^*}^\epsilon z)_U\,ds \right| = \\
& & =\left|\int_\tau^t (e^{-\delta s}u(t-s),e^{\delta s}B^*e^{A^*s}{A^*}^\epsilon z)_U\,ds \right| + \\
& & + \left|\int_0^\tau e^{-\delta s}(u(t-s)-u(\tau-s),e^{\delta s}B^*e^{A^*s}{A^*}^\epsilon z)_U\,ds \right|\le \\
& & \le c(q,r) \left(\|e^{-\delta \cdot}u(t-\cdot)\|_{L^{q'}(\tau,t;U)}+\|e^{-\delta \cdot}(u(t-\cdot)-u(\tau-\cdot))\|_{L^{q'}(0,\tau;U)}\right)\cdot 
\\
& & \qquad \cdot \| e^{\delta \cdot}B^*e^{A^*\cdot}{A^*}^\epsilon z\|_{L^q(0,t;U)} \le \\
& & \le c(q,r) \left( \|u\|_{L^r(0,t-\tau;U)} + \|u(\cdot +t-\tau)-u(\cdot)\|_{L^r(0,\infty;U)}\right)\|z\|_Y\,,
\eeyy
and the result readily follows.

\end{proof}

We recall that the adjoint $L^*$ of the operator $L$ is defined as follows:
\begin{equation} \label{e:L-adjoint}
L^*:v\mapsto L^*v\,, \quad L^*v(s)= \int_s^\infty B^*e^{A^*(t-s)}v(t)\,dt, \qquad s\ge 0\,.
\end{equation}
As it is well known, the regularity analysis of $L^*$ is also central to the study of the optimal control problem.

\begin{proposition}\label{p:stimeL*} 
For the operator $L^*$ defined in \eqref{e:L-adjoint} the following properties hold true.

\begin{description}

\item[(i)] $L^*$ maps continuously $L^1(0,\infty;Y)$ into $L^r(0,\infty;U)$ for each $r\in [1,\frac1{\gamma}[$;

\item[(ii)] $L^*$ maps continuously $L^p(0,\infty;Y)$ into $L^r(0,\infty;U)$ for each $p\in ]1,\frac1{1-\gamma}[$ and $r\in [p,\frac{p}{1-(1-\gamma)p}]$;

\item[(iii)] $L^*$ maps continuously $L^p(0,\infty;Y)$ into $L^r(0,\infty;U)$ for each $p\in [\frac1{1-\gamma},\infty[$ and $r\in [p,\infty[$.

\item[(iv)] $L^*$ maps continuously $L^p(0,\infty;\cD({A^*}^{\epsilon}))$ 
into $L^\infty(0,\infty;U)$ for any $p\in ]\frac1{1-\gamma},\infty]$;

\item[(v)] $L^*$ maps continuously $L^1(0,\infty;[\cD(A^{\epsilon})]')$ into $L^r(0,\infty;U)$ for all $r\in [1,q]$.
\end{description}
\end{proposition}

\begin{proof}
It follows by Proposition~\ref{p:stimeL} by duality: indeed, we have for each $p\ge 1$
\begin{equation*}
\left|\int_0^\infty (L^*v(t),u(t))_U\,dt\right| =\left|\int_0^\infty (v(t),Lu(t))_U\,dt\right| \le \|v\|_{L^p(0,\infty;U)}\|Lu\|_{L^{p'}(0,\infty;U)}\,.
\end{equation*}
Now to prove (i) we take $p=1$, so that $p'=\infty$, and by Proposition~\ref{p:stimeL}(iv) we get
\begin{equation*}
\Big|\int_0^\infty (L^*v(t),u(t))_U\,dt\Big| \le 
c\,\|v\|_{L^1(0,\infty;U)}\|u\|_{L^{r'}(0,\infty;U)} 
\quad \forall r'\in \Big(\frac1{1-\gamma},\infty\Big]\,,
\end{equation*}
which means $L^*v\in L^r(0,\infty;U)$ for each $r\in [1,1/\gamma)$. 
\\
Similarly, to prove (ii) we take $p\in (1,\frac1{1-\gamma})$, so that 
$p'\in (\frac1{\gamma},\infty)$: then, since $p'=\frac{r'}{1-(1-\gamma)r'}$ if and only if 
$r'=\frac{p'}{1+(1-\gamma)p'}$,
Proposition~\ref{p:stimeL}(ii) yields 
\begin{equation*}
\Big|\int_0^\infty (L^*v(t),u(t))_U\,dt\Big| \le 
\, \|v\|_{L^p(0,\infty;U)}\|u\|_{L^{r'}(0,\infty;U)} \ \forall r'\hspace{-1mm}\in \hspace{-1mm}
\Big[\frac{p'}{1+(1-\gamma)p'}\,,p'\Big]\,
\end{equation*}
which means $L^*v\in L^r(0,\infty;U)$ for each $r\in [p,\frac{p}{1-(1-\gamma)p}]$. 
\\
To prove (iii), we take $p\in [\frac1{1-\gamma},\infty)$, so that $p'\in (1,\frac1{\gamma}]$; 
since again $p'=\frac{r'}{1-(1-\gamma)r'}$ if and only if $r'=\frac{p'}{1+(1-\gamma)p'}$, by 
Proposition~\ref{p:stimeL}(i) we obtain
\begin{equation*}
\Big|\int_0^\infty (L^*v(t),u(t))_U\,dt\Big| 
\le c \|v\|_{L^p(0,\infty;U)}\|u\|_{L^{r'}(0,\infty;U)} 
\quad \forall r'\hspace{-1mm}\in\hspace{-1mm} \Big[\frac{p'}{1+(1-\gamma)p'}\,,p'\Big]\,,
\end{equation*}
which means $L^*v\in L^r(0,\infty;U)$ for each $r\in [p,\frac{p}{1-(1-\gamma)p}]$. 
\\
The assertion (iv) is the dual statement of (i) of Proposition~\ref{p:stimeL}.

\smallskip
Finally, (v) follows again from the assertion (v) of Proposition~\ref{p:stimeL}
by duality.
We take $v\in L^1(0,\infty;[\cD(A^\epsilon)]')$, which by \cite[Remark~A.1(ii)]{abl-2} 
means $v(t) = {A^*}^\epsilon w(t)$, with $w\in L^1(0,\infty;Y)$.
We then compute 
\begin{eqnarray*}
\lefteqn{\int_0^\infty (L^*v(t),u(t))_U\,dt = \int_0^\infty ({A^*}^\epsilon w(t),Lu(t))_Y\,dt =} 
\\
& & = \int_0^\infty (w(t),A^\epsilon Lu(t))_Y \, dt 
\le \|w\|_{L^1(0,\infty;Y)}\, \|Lu\|_{C_b([0,\infty[;\cD(A^\epsilon))} 
\\
& & \le c_r \|{A^*}^{-\epsilon}v\|_{L^1(0,\infty;Y)}\, \|u\|_{L^r(0,\infty;U)}\,, 
\qquad r\in [q',\infty]\,,
\end{eqnarray*}
where in the latter estimate we used (v) of Proposition~\ref{p:stimeL}.
The above shows that $L^*$ maps $L^1(0,\infty;[\cD(A^\epsilon)]')$ continuously 
in the dual space of $L^r(0,\infty;U)$, for any $r\in [q',\infty]$, that is
$L^{r'}(0,\infty;U)$ with $r'\in [1,q]$, thus concluding the proof.
\end{proof}

\begin{remark} \label{r:extend-with-exponents}
\begin{rm}
All the regularity results provided by the statements contained in Proposition~\ref{p:stimeL}
and Proposition~\ref{p:stimeL*} may be easily extended to natural analogs in $L^p_\delta$-spaces, 
if $\delta\in [0,\omega\wedge \eta)$ is given, maintaining the respective summability exponents 
$p$. The proof is omitted.
\end{rm}
\end{remark} 

In the following Proposition we collect a couple of regularity results which---in view
of Remark~\ref{r:extend-with-exponents}---are in essence contained in assertion (v) of 
Proposition~\ref{p:stimeL}.
However, since \eqref{e:both-crucial} below will play a crucial role in 
the proof of well-posedness of the Algebraic Riccati equations, its statement
is given explicitly, along with a distinct proof.

\begin{proposition}\label{p:L*R*R} 
For any $\delta\in [0,\omega\wedge \eta[$, the following regularity results are valid:
\begin{subequations}\label{e:both-crucial}
\begin{align}
& R^*RL \in \cL(L_{\delta}^{q'}(0,\infty;U),L_{\delta}^\infty(0,\infty;\cD({A^*}^{\epsilon})))\,;
\label{e:for-crucial}\\[1mm]
& L^*R^*R \in 
\cL(L_{\delta}^1(0,\infty;[\cD({A^*}^{\epsilon})]'),L_{\delta}^q(0,\infty;U))\,.
\label{e:crucial}
\end{align}
\end{subequations}

\end{proposition}

\begin{proof}
Let $u\in L_{\delta}^{q'}(0,\infty;U)$ and $z\in \cD(A^{\epsilon})$.
We have 
\beyy 
\lefteqn{
\hspace{-5mm}|(R^*RL u(t),A^{\epsilon}z)_Y| = 
\Big| \int_0^t (u(\sigma),B^*e^{A^*(t-\sigma)}R^*RA^{\epsilon}z)_U\, d\sigma\Big| \le} 
\\
& & \le \int_0^t \Big|\big(e^{-\delta(t-\sigma)}u(\sigma),e^{\delta(t-\sigma)}B^*e^{A^*(t-\sigma)}R^*RA^{\epsilon}z\big)_U \Big|\,dt \le 
\\
& & \le C\, e^{-\delta t}\,\|e^{\delta\cdot}u\|_{L^{q'}(0,\infty;U)}\,
\|e^{\delta \cdot}B^*e^{A^*\cdot}R^*RA^\epsilon z\|_{L^q(0,\infty;U)}\le 
\\
& & \le C\, e^{-\delta t}\, 
\|e^{\delta\cdot}u\|_{L^{q'}(0,\infty;U)}\,\|z\|_Y\,,
\eeyy
where $\delta \in ]0,\omega\wedge \eta[$ and in the last estimate we used 
Proposition~\ref{p:expBA}.
Therefore,
\begin{equation*}
|(e^{\delta t}R^*RL u(t),A^{\epsilon}z)_Y| \le
C\, \|e^{\delta\cdot}u\|_{L^{q'}(0,\infty;U)}\, \|z\|_Y\,,
\end{equation*}
i.e. ${A^*}^{\epsilon}e^{\delta t}R^*RL u \in L^{\infty}(0,\infty;U)$ and
\beq\label{e:stimaR*RL}
\|R^*RL u(\cdot)\|_{L_{\delta}^\infty(0,\infty;\cD({A^*}^{\epsilon}))} 
\le c\, \|u\|_{L_{\delta}^{q'}(0,\infty;U)}\,.
\eeq 

The statement in \eqref{e:crucial} follows from \eqref{e:for-crucial} by duality.

\end{proof}

%------------------------------------------------------------------
\section{The optimal state semigroup, the optimal cost operator}
\label{s:gain}
%------------------------------------------------------------------

\subsection{The optimal pair % Statements~S1.~and~S2.~of Theorem~\ref{t:main}
}
Since we are dealing with a classical linear-quadratic problem, the existence of a unique optimal control minimizing the cost functional \eqref{e:funct} follows by convex optimization arguments.
In addition, the Lagrange multipliers method yields the optimality condition for the optimal pair 
$(\hat{y},\hat{u})\in L^2(0,\infty;Y)\times L^2(0,\infty;U)$, that is
\beq\label{e:optimality-cond} 
\hat{u}= -L^*R^*R\hat{y}\,, 
\eeq
where 
\beq \label{e:optimal-state} 
\hat{y}=e^{A\cdot}y_0 + L\hat{u}\,.
\eeq
Then, owing to the regularity provided by Propositions~\ref{p:stimeL} and \ref{p:stimeL*}, as well as to the decay assumption on the semigroup $e^{At}$, we can appeal to a classical bootstrap process: 
we set
\begin{equation*}
p_0=2, \qquad p_{n+1}=\frac{p_n}{1-(1-\gamma)p_n}, \quad 0\le n <N\,,
\end{equation*}
where $N\in \N$ is the first positive integer such that $p_N > \frac{1}{1-\gamma}$; 
such an integer does exist because
\begin{equation*}
p_{n+1}-p_n = p_n \, \frac{(1-\gamma)p_n}{1-(1-\gamma)p_n} >
\frac{4(1-\gamma)}{2\gamma-1}>0\,.
\end{equation*}
Since $\hat{u}\in L^{p_0}(0,\infty;U)$, by \eqref{e:optimal-state} 
we know $\hat{y}\in L^{p_1}(0,\infty;Y)$, so that we deduce as a first step
$\hat{u}(\cdot,s;x)\in L^r(0,\infty;U)$ for all $r<p_2$ as well as, by 
\eqref{e:optimality-cond}, $\hat{y}\in L^s(0,\infty;Y)$ for all $s<p_3$. 
Thus, after $n$ steps, we get 
\begin{equation*}
\hat{u}\in L^r(0,\infty;U) \quad \forall r<p_{2n}\,, 
\qquad \hat{y}\in 
L^s(0,\infty;Y) \quad \forall s<p_{2n+1}\,.
\end{equation*}

This procedure stops as soon as $2n$ or $2n+1$ equals $N$: indeed, if $N=2n$ we get at the $n$-th step 
$\hat{u} \in L^r(s,T:U)$ for all $r<p_N$ and $\hat{y}\in C_b([0,\infty[;Y)$, so that in the next step we obtain 
$\hat{u}\in L^p(0,\infty;U)$ for all $p<\infty$; if $N=2n+1$ we
find directly, at the $(n+1)$-th step, $\hat{u}\in L^p(0,\infty;U)$ for all $p<\infty$ and 
$\hat{y}\in C_b([0,\infty[;Y)$. 
At any step we also find the corresponding bound in terms of $\|x\|_Y\,$. 
Thus, we have proved the following Proposition.

\begin{proposition}[Statement~S1.~of Theorem~\ref{t:main}] \label{p:reg-uy} 
For any $x\in Y$ there exists a unique optimal control $\hat{u}(\cdot)$ for 
problem \eqref{e:state-eq}-\eqref{e:funct}.
The optimal pair $(\hat{y},\hat{u})$, which a priori belongs to 
$L^2(0,\infty;Y)\times L^2(0,\infty;U)$, further satisfies
\begin{equation*}
\hat{y}\in L^p(0,\infty;Y)\cap C_b([0,\infty);Y)\,, \quad \hat{u}\in L^p(0,\infty;Y)
\qquad \forall p\in [2,\infty)\,,
\end{equation*}
continuously with respect to $x\in Y$:
\begin{subequations}
\begin{align} 
& \|\hat{y}\|_{L^p(0,\infty;U)}\le c_p \|x\|_Y \quad \forall p \in [2,\infty)\,;  \qquad 
 \|\hat{y}\|_{C_b([0,\infty[;Y)} \le c\|x\|_Y\,,
\label{e:continuously-state} \\[2mm]
& \|\hat{u}\|_{L^p(0,\infty;U)}\le c_p \|x\|_Y \quad \forall p \in [2,\infty)\,.
\label{e:continuously-control}
\end{align}
\end{subequations}
\end{proposition}

\smallskip
Let us denote by $\Phi(\cdot)x$ the optimal state corresponding to an initial state $x\in Y$
and to the optimal control $\hat{u}(\cdot)$:
\begin{equation*} %\label{e:optimal-state-semigroup}
\Phi(t)x:=\hat{y}(t)=y(t,x;\hat{u})\,.
\end{equation*}
The significant basic properties of the family of operators $\{\Phi(t)\}_{t\ge 0}$ are 
briefly recorded below.

\smallskip
\noindent
1. The estimate on the right in \eqref{e:continuously-state} of Proposition~\ref{p:reg-uy} yields
$\|\Phi(t)\|_{\cL(Y)}<\infty$ uniformly in $t\ge 0$. 

\smallskip
\noindent
2. It is not difficult the show that $\hat{y}(t,x)$ possesses the transition property 
\begin{equation} \label{e:transition}
\hat{y}(t+\sigma;x)=\hat{y}(\sigma;\hat{y}(t,x))\in C_b([0,\infty);Y)
\qquad x\in Y\,,\; t,\sigma> 0\,,
\end{equation}
which renders $\Phi(t)$ a strongly continuous semigroup on $Y$; indeed, 
$\Phi(0)=I$ is readily checked, as $\hat{y}(0,x)=x$ for any $x\in Y$.
We just note that the proof of \eqref{e:transition} follows a standard route: 
it is based on the identity 
\begin{equation}\label{e:introduces-Lambda}
\hat{y}(t;x) + [LL^*R^*R\hat{y}(\cdot,x)](t)= e^{At}x\,, 
\end{equation}
which follows from $L\hat{u}= -LL^*R^*R\hat{y}$ (this, in turn, is a direct consequence 
of the optimality condition \eqref{e:optimality-cond}), combined with \eqref{e:optimal-state}.
(See, e.g., \cite[Theorem~6.25.1, p.~626]{las-trig-encyclopedia} for an outline of the necessary 
steps, though in the context of min-max game theory.)

Thus, \eqref{e:transition} confirms that $\Phi(t)$ is a $C_0$-semigroup on $Y$,
which introduces its infinitesimal generator, that is the linear, closed operator 
$A_P:\cD(A_P)\subset Y\to Y$ defined by
\begin{equation}
\begin{aligned}
\cD(A_P) &:= \big\{x\in Y: \:\lim_{t\to 0^+ } \frac{\Phi(t)x-x}{t} \quad\textrm{exists in}\; Y\big\}\,,
\\
A_Px &:=\lim_{t\to 0^+ } \frac{\Phi(t)x-x}{t}\,,
% = {\frac{d}{dt}}^+\Phi(t)x
\qquad\qquad  x\in \cD(A_P)\,.
\label{e:optimal-state-generator}
\end{aligned}
\end{equation}
The domain $\cD(A_P)$ is by its nature dense.

\smallskip
\noindent
3. Since $\Phi(\cdot)x\in L^2(0,\infty;Y)$ and $\Phi(t)$ is a $C_0$-semigroup, 
then according to Datko's Theorem $\Phi(t)=e^{A_Pt}$ is further exponentially stable,
which confirms \eqref{e:exponential_1} .

\smallskip
In summary, we have the following 
\begin{proposition}[Statement~S2.~of Theorem~\ref{t:main}] \label{p:optimal-state-semigroup}
The optimal state $\hat{y}(t,x)$ defines a strongly continuos semigroup $e^{A_Pt}$ in $Y$, 
$t\ge 0$, which is exponentially stable.
\end{proposition}
 
A deeper insight into the additional differential properties of the optimal state semigroup 
$\Phi(t)=e^{A_Pt}$---the ones which will ultimately allow us to establish well-posedness of the 
ARE---is given in Section~\ref{s:are}. 
% (specifically, its being differentiable on $\cD(A)$ (besides on $\cD(A_P)$))---

%------------------------------------------------------------------
\subsection{The Riccati operator $P$}
%------------------------------------------------------------------

The Riccati operator $P$ is initially introduced, as usual, explicitly in terms of the optimal state semigroup $\Phi(t)$:
\begin{equation}\label{e:definition-of-riccatiop}
Px :=\int_0^\infty e^{A^*t}R^*R\Phi(t)x\,dt\,, \qquad x\in Y\,.
\end{equation}

We record first the basic properties of the operator $P$.
% making reference to the existing literature.

\begin{proposition} 
The operator $P$ belongs to $\cL(Y)\cap \cL(\cD(A^\epsilon),\cD({A^*}^\epsilon))$.
\end{proposition}

\begin{proof}
With $x\in Y$ we just utilize \eqref{e:stability} of Hypothesis~\ref{h:ip1} and 
the continuity property \eqref{e:continuously-state} of Proposition~\ref{p:reg-uy} to find
\beyy
\|Px\|_Y & \le & \int_0^\infty \|e^{A^*t}\|_{\cL(Y)} \|R^*R\|_{\cL(Y)}\|\Phi(t)x\|_Y \,dt \le \\
& \le & \int_0^\infty Mc\,e^{-\omega t}\,dt\, \|x\|_Y \le c'\|x\|_Y\,.
\eeyy

Next, for $x\in \cD(A^\epsilon)$ we write $\Phi(t) = e^{At}x + L\hat{u}(t)$ and employ
the assumption \eqref{e:keyasR} on the observation operator $R$, the regularity provided by 
\eqref{e:stimaR*RL} and \eqref{e:continuously-state} % of Proposition~\ref{p:reg-uy} 
to obtain
\beyy
\|Px\|_{\cD({A^*}^\epsilon)} & \le &
\int_0^\infty \|e^{A^*t}\|_{\cL(Y)}\|R^*R[e^{At}x + L\hat{u}(t)]\|_{\cD({A^*}^\epsilon)}\,dt 
\\
& \le & \int_0^\infty M\,e^{-\omega t}
\|{A^*}^\epsilon R^*RA^{-\epsilon}\|_{\cL(Y)}\|e^{At}A^\epsilon x\|_Y\,dt 
\\
& & \quad + \int_0^\infty M\,e^{-\omega t}
\|R^*RL\hat{u}\|_{L^\infty(0,\infty;\cD({A^*}^\epsilon))}\,dt  
\\
& \le & c \,\big[ \|x\|_{\cD(A^\epsilon)}+ \|\hat{u}\|_{L^{q'}(0,\infty;U)}\big] 
\le c\|x\|_{\cD(A^\epsilon)}\,.
\eeyy
\end{proof}

By the very definition of $P$ in \eqref{e:definition-of-riccatiop} it easily
follows---using the optimality condition \eqref{e:optimality-cond}---that the optimal cost is 
a quadratic form on $Y$, that is $(Px,x)$; this motivates the term {\em optimal cost} operator 
used for $P$. 
The proof of this fact is briefly recorded below.

\smallskip
We simply compute, for any $x\in Y$, the cost corresponding to the optimal pair,
that is 
\begin{align*}
J(\hat{u})&=\int_0^\infty \Big\{\big(R \hat{y}(t,x), R \hat{y}(t,x)\big)_Y
+ \big(\hat{u}(t,x),\hat{u}(t,x)\big)_U\Big\}\, dt
\\
&= \int_0^\infty \big(R \hat{y}(t,x),R\hat{y}(t,x)\big)_Y\,dt
- \int_0^\infty \big(L^*R^* R\hat{y}(\cdot\,,x)(t),\hat{u}(t,x)\big)_U\,dt
\\
&= \int_0^\infty \big(R \Phi(t) x,R[e^{At}x+(L\hat{u})(t)]\big)_Y\,\,dt-
\int_0^\infty \big(R\Phi(t)x,R(L\hat{u})(t)\big)_Y\,\,dt
\\
&= \int_0^\infty (R \Phi(t) x,Re^{At}x)_Y \,dt=
\Big(\int_0^\infty e^{A^*t}R^*R \Phi(t) x\,dt,x\Big)_Y\,.
\end{align*}
Owing to \eqref{e:definition-of-riccatiop}, the last equality shows that 
\begin{equation}\label{e:quadratic-form} 
J(\hat{u})\equiv (Px,x)_Y \myspace \forall x\in Y\,.
\end{equation}

\smallskip
\begin{remark}
\begin{rm}
In much the same way one gets $J(\hat{u})\equiv (x,Px)_Y$, which combined with 
\eqref{e:quadratic-form} shows that $P$ is a self-adjoint operator;
moreover, since $J(\hat{u})\ge 0$ for any $x\in Y$, \eqref{e:quadratic-form} establishes
as well that $P$ is non-negative.
\end{rm}
\end{remark}

Similarly, one may show that 
\begin{equation*} 
(Px_1,x_2)_Y =\int_0^\infty \Big\{\big(R \hat{y}(t,x_1),R \hat{y}(t,x_2)\big)_Y
+ \big(\hat{u}(t,x_1),\hat{u}(t,x_2)\big)_U\Big\}\, dt\\, \; x_1,x_2\in Y\,.
\end{equation*}

\medskip
We now introduce the function $p(t,x)$ defined by
\begin{equation}\label{e:definition-of-pfunction}
p(t,x):= \int_t^\infty e^{A^*(\tau-t)}R^*R \hat{y}(\tau,x)\, d\tau
\qquad t>0\,, \; x\in Y\,, 
\end{equation}
which is the unique solution to the Cauchy problem
\begin{equation}
\begin{cases}
p'(t;x)= -A^*p(t)-R^*R \hat{y}(t,x)
\\
\displaystyle\lim_{t\to +\infty} p(t;x)=0\,.
\end{cases}
\end{equation}

A classical and elementary computation shows that if we rewrite $P$ in 
\eqref{e:definition-of-riccatiop} as a time-dependent function, that is
\begin{equation*}
Px :=\int_t^\infty e^{A^*(\tau-t)}R^*R\Phi(\tau-t)x\,dt\,, \qquad x\in Y\,, \; t\ge 0\,,
\end{equation*}
and next apply the above formula when $x$ is replaced by $\Phi(t)x$, namely:
\begin{equation*}
P\Phi(t)x :=\int_t^\infty e^{A^*(\tau-t)}R^*R\Phi(\tau-t)\Phi(t)x\,dt
=\int_t^\infty e^{A^*(\tau-t)}R^*R\Phi(\tau)x\,dt\,, 
\end{equation*}
then we immediately see that
\begin{equation}\label{e:link-riccatiop-functionp}
P\Phi(t)x\equiv p(t;x) \qquad \forall t\ge 0\,, \; x\in Y\,, 
\end{equation} 
which establishes a relation between the Riccati operator $P$ and the function $p$. 

The above formula \eqref{e:link-riccatiop-functionp} is the starting point in order to derive the {\em feedback representation} of the optimal control---a property which is central to solvability of Problem~\ref{p:problem-0}.
Indeed, if we return to the optimality condition \eqref{e:optimality-cond}, and write 
explicitly the integral operator $L^*$, we know that 
\begin{equation} \label{e:true}
\hat{u}(t,x)=-\int_t^\infty B^*e^{A^*(\tau-t)}R^*R\Phi(\tau)x\,d\tau \qquad\forall x\in Y\,.
\end{equation}
Thus, a {\em formal} computation yields    
\begin{equation} 
\hat{u}(t,x)=-B^*\int_t^\infty e^{A^*(\tau-t)}R^*R\Phi(\tau)x\,d\tau=-B^*p(t;x)\,, 
\label{e:to-be-justified}
\end{equation} 
which combined with \eqref{e:link-riccatiop-functionp} would imply
\begin{equation*}
\hat{u}(t,x) = - B^*P\Phi(t)x \qquad \textrm{for any $x\in Y$ and for a.e. $t>0$.}
\end{equation*}
However, going from \eqref{e:true} to \eqref{e:to-be-justified} necessitates a  % more subtle
deeper technical justification, which is found in the next section.  

%-----------------------------------------------------------------------------
\subsection{The gain operator $B^*P$}
%-----------------------------------------------------------------------------
The technical issue raised in the previous section---namely, the need for a rigorous 
justification of the first equality in \eqref{e:to-be-justified}---as well as our next 
and major task (that is to show that the optimal cost operator $P$ does satisfy the algebraic Riccati equation  % \eqref{e:riccati-eq} 
corresponding to the optimal control Problem~\ref{p:problem-0}), 
require that we are able to give a proper meaning to the gain operator $B^*P$. 
We accomplish this goal by introducing a linear operator, which will eventually coincide 
with $B^*P$, that is shown to be bounded on a dense subset of $Y$ (and yet unbounded on $Y$).
The regularity result for the operator $R^*RL$ set forth in Proposition~\ref{p:L*R*R}, along 
with Proposition~\ref{p:expGf} provide the tools.
\begin{theorem}[Statement S4.~of Theorem~\ref{t:main}] \label{t:bounded-gain} 
Let $\epsilon$ be such that Hypothesis~\ref{h:ip2}(iii) holds true.
Then, the following statements are valid.
\begin{enumerate}
\item[(i)] 
The integral 
\begin{equation} \label{e:def-T}
\cT x :=\int_0^\infty B^*e^{A^*t}R^*R\Phi(t)x\,dt
\end{equation}
defines a (linear) bounded operator from $\cD(A^\epsilon)$ into $U$. 
\item[(ii)]
The operator $B^*P$, formally defined by
\begin{equation*}
B^*Px :=B^*\int_0^\infty e^{A^*t}R^*R\Phi(t)x\,dt\,,
\end{equation*}
% which makes sense at least for $x\in $\cD(A^*)$?, 
coincides with $\cT$ on $\cD(A^\epsilon)$, and hence
\begin{equation} \label{e:b*p-bdd}
B^*P \in \cL(\cD(A^\epsilon),U)\,.
\end{equation}
\end{enumerate}
\end{theorem}

\begin{proof}
(i) Let $x\in \cD(A^\epsilon)$. 
According to the decomposition \eqref{e:key-hypo} for $B^*e^{A^*t}$, the integrand 
in \eqref{e:def-T} is split as follows:
\beq \label{e:split}
B^*e^{A^*t}R^*R\Phi(t)x = F(t)R^*R\Phi(t)x + G(t)R^*Re^{At}x + G(t)R^*RL\hat{u}(t)\,,
\eeq
with a further splitting due to \eqref{e:optimal-state}.
We will show that each summand on the right hand side of \eqref{e:split}
is a function in $L^1(0,\infty;U)$.

First, the basic Hypothesis~\ref{h:ip2}(i) combined with the regularity 
\eqref{e:continuously-state} of the optimal state % in Proposition~\ref{p:reg-uy}
readily imply 
\begin{equation*}
\|F(t)R^*R\Phi(t)x\|_U \le N\,t^{-\gamma}\,e^{-\eta t}\|R^*R\|_{\cL(Y)} \|\Phi(t)x\|_Y
\le c\,t^{-\gamma}\,e^{-\eta t}\,\|x\|_Y\,,
\end{equation*}
which shows
\beq\label{e:magg1}
\int_0^\infty \|F(t)R^*R\Phi(t)x\|_U \,dt \le c \|x\|_Y \qquad \forall x\in Y\,.
\eeq

% 2. 
Next, we recall Proposition~\ref{p:expGf} and utilize \eqref{p:expGf} with $p=1$, 
$f(t)=R^*Re^{At}x$ and $r=\infty$, where $R^*R$ is subject to Hypothesis~\ref{h:ip2}(iii)(b), 
thus obtaining 
\beyy
\int_0^\infty \|G(t)R^*Re^{At}x\|_U\,dt 
& \le & c \,\sup_{t>0} \|R^*Re^{At}x\|_{\cD({A^*}^\epsilon)} \le 
\\
& \le & C\,\|{A^*}^\epsilon R^*R{A^*}^{-\epsilon}\|_{\cL(Y)}
\|x\|_{\cD(A^\epsilon)}\,,
\eeyy
so that
\beq\label{e:magg2}
\int_0^\infty \|G(t)R^*Re^{At}x\|_U\,dt 
\le c\|x\|_{\cD(A^\epsilon)} \qquad \forall x\in \cD(A^\epsilon)\,.
\eeq

% 3. 
Finally, for the third summand in \eqref{e:split} we apply once more 
Proposition~\ref{p:expGf}, this time with $f(t) =R^*RL\hat{u}(t)$.
Notice that the membership $R^*RL\hat{u}(\cdot)\in L^\infty(0,\infty;\cD({A^*}^\epsilon))$ follows from \eqref{e:for-crucial}, in view of the regularity established 
in Proposition~\ref{p:reg-uy}, which provide as well the appropriate estimate.
Therefore,
\begin{equation*}
\int_0^\infty \|G(t)R^*RL\hat{u}(t)\|_U\,dt 
\le c\, \sup_{t>0} \|R^*RL\hat{u}(t)\|_{\cD({A^*}^\epsilon)} 
\le c \|\hat{u}\|_{L^{q'}(0,\infty;U)}\,,
\end{equation*}
which yields, in view of \eqref{e:continuously-control},
\beq\label{e:magg3}
\int_0^\infty \|G(t)R^*RL\hat{u}(t)\|_U\,dt \le c\|x\|_Y \qquad \forall x\in Y\,.
\eeq

% 4. 
The obtained estimates \eqref{e:magg1}, \eqref{e:magg2} and \eqref{e:magg3} 
show that
\begin{equation*}
\|\cT x\|_U = \Big\|\int_0^\infty B^*e^{A^*t}R^*R\Phi(t)x\,dt\Big\|_U 
\le c \,\|x\|_{\cD(A^\epsilon)}\qquad \forall x\in \cD(A^\epsilon)\,,
\end{equation*}
that is $\cT\in \cL(\cD(A^\epsilon),U)$. 
 
\medskip
\noindent
(ii) We return to the decomposition \eqref{e:key-hypo}, where 
by Hypotheses~\ref{h:ip2}(i) and (iii)(a) we know
\begin{equation*}
F(t)\in \cL(Y,U), \qquad G(t)\in \cL(\cD({A^*}^\epsilon),U)\,,
\end{equation*}
respectively.
On the other hand, we also have for $t>0$
\beq\label{e:rappB2}
B^*e^{A^*t}x = [e^{At}B]^*x \qquad \forall x\in \cD(A^*)\,,
\eeq
where 
\begin{equation*}
[e^{At}B]^*\in \cL(\cD({A^*}^\epsilon),U)\,;
\end{equation*}
see \cite[Lemma~A.2]{abl-2}.
Thus, the operators $F(t)+G(t)$ and $[e^{At}B]^*$ are well defined 
on $\cD({A^*}^\epsilon)$ and both coincide with $B^*e^{A^*t}$ on $\cD(A^*)$. 
Hence, if $x\in \cD(A^\epsilon)$ we can write for all $v\in U$
\begin{align}
(\cT x,v)_U & = \int_0^\infty ([e^{At}B]^*R^*R\Phi(t)x,v)_U\,dt = 
\nonumber\\
& = \int_0^\infty (R^*R\Phi(t)x,e^{At}Bv)_{\cD({A^*}^\epsilon),[\cD({A^*}^\epsilon)]'}\,dt =
\nonumber\\
& = \lim_{s\to 0^+}\int_0^\infty (e^{A^*s}R^*R\Phi(t)x,e^{At}Bv)_{\cD({A^*}^\epsilon),[\cD({A^*}^\epsilon)]'}\,dt =
\nonumber\\ 
& = \lim_{s\to 0^+}\int_0^\infty (e^{A^*t}R^*R\Phi(t)x,e^{As}Bv)_{\cD({A^*}^\epsilon),[\cD({A^*}^\epsilon)]'}\,dt =
\nonumber\\
& = \lim_{s\to 0^+} (Px,e^{As}Bv)_{\cD({A^*}^\epsilon),[\cD({A^*}^\epsilon)]'}
= \lim_{s\to 0^+} ([e^{As}B]^*Px,v)_U
\label{e:weak-limit-for-calt}\\
& =: (B^*Px,v)_U\,.
\nonumber 
\end{align}
This shows that for all $x\in \cD(A^\epsilon)$ the operator $B^*P$ is uniquely defined as a weak limit in $U$, and coincides with $\cT$, i.e. $B^*P x = \cT x$.
By part (i), $B^*P\equiv \cT \in \cL(\cD(A^\epsilon),U)$ which concludes the
proof of (ii).
% ; in particular, formula \eqref{e:to-be-justified} is now justified.  
\end{proof}

\begin{remark}
\begin{rm}
To give a deeper insight into the previous result, we summarize the major steps
of its proof, complemented by few additional technical remarks.
First, the operator $\cT$ defined by \eqref{e:def-T} is shown to be a bounded operator 
from $\cD(A^\epsilon)$ in $U$.
Next, the equivalence between $(\cT x,v)_U$ and the limit in \eqref{e:weak-limit-for-calt}
% that $(\cT x,v)_U = \lim_{s\to 0^+} ([e^{As}B]^*Px,v)_U$ 
(for any $x\in \cD(A^\epsilon)$ and any $v\in U$) shows that $\cT x$ coincides  
with the weak limit $\wlim_{s\to 0^+}[e^{As}B]^*Px$ in $U$.
Notice that the operator $[e^{As}B]^*$ is bounded from $D({A^*}^\epsilon)$ in $U$, 
and yet its norm blows up, as $s\to 0^+$, like $s^{-\gamma}$; see \cite[Lemma~A.2]{abl-2}.

Finally, the motivation for denoting by $B^*Px$ the aforesaid weak limit is the following. 
When $Px$ belongs to $\cD(A^*)$ rather than just to $\cD({A^*}^\epsilon)$, we know that 
$[e^{As}B]^*Px = B^*e^{A^*s}Px \to B^*Px$,
because $e^{A^*s}Px \to Px$ in $\cD(A^*)$ as $s\to 0^+$, whereas by its very definition 
$B^*\in \cL(\cD(A^*),U)$.
We note that the membership $Px\in \cD(A^*)$ does hold, at least for $x\in \cD(A_P)$, 
as we will see later in Lemma~\ref{l:riccati-op}. 
\end{rm}
\end{remark}

%-----------------------------------------------------------------------------
\subsection{The feedback representation of the optimal control}
%-----------------------------------------------------------------------------
On the basis of the analysis carried out in the previous section, we would intend now 
to derive the feedback representation of the optimal control.
The obtained property \eqref{e:b*p-bdd} for the gain operator suggests that we preliminarly 
assume $x\in\cD(A^\epsilon)$ and show the validity of the feedback formula for these smoother initial data.
In fact, with $x\in\cD(A^\epsilon)$, we can exploit the additional regularity of the 
optimal state, which is made clear in the following Proposition, also of intrinsic interest;
the proof is postponed to Appendix~\ref{a:lonely}. 

\begin{proposition}[Statement~S8.~of Theorem~\ref{t:main}] \label{p:statement-eight}
If $x\in\cD(A^\epsilon)$, the optimal state $\Phi(t)x$ belongs to $\cD(A^\epsilon)$ 
for all $t\ge 0$. More precisely, 
\begin{equation} \label{e:goal_1}
x\in\cD(A^\epsilon) \Longrightarrow 
\Phi(\cdot\,)x\in L^p(0,\infty;\cD(A^\epsilon))\cap C_b([0,\infty);\cD(A^\epsilon))
\end{equation}
for all $p\in [1,\infty]$, continuously with respect to $x$.
Consequently, 
\begin{equation} \label{e:goal_2}
x\in\cD(A^\epsilon) \Longrightarrow \hat{u}\in C_b([0,\infty);U)\,.
\end{equation}
\end{proposition}
  
\smallskip
Assuming $x\in\cD(A^\epsilon)$, by Proposition~\ref{p:statement-eight} we know that 
$\Phi(\cdot\,)x\in L^p(0,\infty;\cD(A^\epsilon))$ and we can appeal to the same arguments used 
in the proof of Theorem~\ref{t:bounded-gain} to give a rigorous justification to \eqref{e:to-be-justified}.
It is now true that \eqref{e:link-riccatiop-functionp}, combined with \eqref{e:to-be-justified}, provides the feedback representation of the optimal control, initially for $x\in\cD(A^\epsilon)$: 
\begin{equation*} 
\hat{u}(t,x) = - B^*P\Phi(t)x \qquad \textrm{for any $x\in \cD(A^\epsilon)$ and for a.e. $t>0$.}
\end{equation*}
This formula is easily extended to all of $x\in Y$, as shown in the Proposition below.
 
\begin{proposition}[Statement S9.~of Theorem~\ref{t:main}] \label{p:feedback}
The (pointwise in time) feedback representation of the optimal control 
\begin{equation} \label{e:feedback}
\hat{u}(t,x) = - B^*P\Phi(t)x\,, \qquad \textrm{for a.e. $t>0$,}
\end{equation}
is valid for any initial state $x\in Y$.
\end{proposition}

\begin{proof}
We know that the relation \eqref{e:feedback} holds for a.e. $t>0$ and for any 
$x\in \cD(A^\epsilon)$. 
Thus, recalling the (continuous with respect to $x$) estimate 
\eqref{e:continuously-control} for the optimal control, we obtain 
\begin{equation}
\|B^*P\Phi(\cdot)\|_{L^p(0,\infty;U)}=\|\hat{u}(\cdot)\|_{L^p(0,\infty;U)}\le C\|x\|_Y
\qquad \forall x\in \cD(A^\epsilon)\,.
\end{equation}
By density, the operator $B^*P\Phi(\cdot)$ is extended to a bounded operator from $Y$ to
$L^p(0,\infty;U)$. Consequently, 
\begin{equation*} 
\hat{u}(t,x) = - B^*P\Phi(t)x \qquad \textrm{for any $x\in Y$ and for a.e. $t>0$.}
\end{equation*}
i.e. the feedback formula \eqref{e:feedback} is extended as well to all $x\in Y$. 
\end{proof}

%---------------------------------------------------------------------------------------------
\section{Towards well-posedness of the Algebraic Riccati Equation} \label{s:are}
%---------------------------------------------------------------------------------------------
In this section we show that the optimal cost operator $P$ introduced in 
Section~\ref{s:gain} does solve the ARE \eqref{e:riccati-eq} corresponding to the optimal 
control Problem~\ref{p:problem-0}.
Owing to formula \eqref{e:definition-of-riccatiop}, this issue is strongly related to 
certain differential properties of the optimal state semigroup $\Phi(t)$, which are 
discussed in Section~\ref{ss:differentiate-Phi(t)B}, culminating with the statement 
of Corollary~\ref{c:differentiability-on-domain_A}.

We will see that a major challenge arises on the operator-theoretic side, as the mapping  
$\Lambda=I+L^*R^*RL$---which, like in previous theories, % of the LQ-problem 
is an {\em isomorphism} on $L^2(0,\infty;U)$---fails to be an {\em isomorphism} in 
$L^q(0,\infty;U)$ or in $L^q_{\delta}(0,\infty;U)$, whilst it is required to admit a bounded inverse acting at least on the latter space.
A solution to this question is given by the distinct result established with
Theorem~\ref{t:invertible-on-X_q}, based on Lemma~\ref{l:invertible-on-L2delta}.
 
\smallskip
To begin with, we provide a preliminary description of the  % domain $\cD(A_P)$ of 
the generator $A_P$ of the optimal state semigroup $\Phi(t)$. 
  
%----------------------------------------------------------------------------------------------
\subsection{The optimal state generator $A_P$. Basic facts}
%----------------------------------------------------------------------------------------------
Let $A_P:\cD(A_P)\subset Y\to Y$ be the (optimal state) generator defined by 
\eqref{e:optimal-state-generator}, i.e. the infinitesimal generator of the strongly continuous semigroup $\Phi(t)$.
With $x\in \cD(A_P)$, we know that 
\begin{equation*}
\exists \lim_{t\to 0^+ }\Big(\frac{\Phi(t)x-x}{t},z\Big) =:(A_Px,z) \qquad \forall z\in Y\,,
\end{equation*}
which---by using the representation of the optimal state $\Phi(t)x$ in terms of the optimal control $\hat{u}(t,x)$---is readily rewritten as follows
\begin{equation*}
\exists \lim_{t\to 0^+ }\Big(\frac{e^{At}x-x}{t}
+\frac1{t}\,\int_0^t e^{A(t-\tau)}B\hat{u}(\tau;x)\,d\tau,z\Big) 
=(A_Px,z) \qquad \forall z\in Y\,.
\end{equation*}
Still with $x\in \cD(A_P)$ and taking now $z\in \cD(A^*)$, we find   
\begin{align*}
& \frac1{t}(\Phi(t)x-x,z)_Y =
\\
& \qquad = \frac1{t}([e^{At}-I]A^{-1}x,A^*z)_Y+\frac1{t}\Big(\int_0^t e^{A(t-\tau)}A^{-1}B\hat{u}(\tau,x)\,d\tau,A^*z\Big)_Y\,,
\end{align*}
that is, by setting $w=A^*z$,
\begin{equation} \label{e:sum-of-limits}
\begin{aligned}
& \frac1{t}(\Phi(t)x-x,(A^*)^{-1}w)_Y=
\\
& \qquad =\frac1{t}([e^{At}-I]A^{-1}x,w)_Y+\frac1{t}\Big(\int_0^t e^{A(t-\tau)}A^{-1}B\hat{u
}(\tau,x)\,d\tau,w\Big)_Y\,.
\end{aligned}
\end{equation}
Thus, if we let $t\to 0^+$ in \eqref{e:sum-of-limits}, we see that the left hand side tends to 
$(A_Px,(A^*)^{-1}w)_Y$, whilst the first summand in the right hand side converges to $(x,w)_Y$. 
As for the second summand in \eqref{e:sum-of-limits}, we employ the feedback representation of the optimal control \eqref{e:feedback} and observe that $x\in \cD(A_P)$ yields
$\Phi(\cdot)x\in C([0,T],\cD(A_P))$, which implies $P\Phi(\cdot)x\in C([0,T],\cD(A^*))$,
since $P\in \cL(\cD(A_P),\cD(A^*))$ (this fact will be shown later: see statement (i) of Lemma~\ref{l:riccati-op}).
Consequently, $B^*P\Phi(\cdot)x\in C([0,T],U)$ and the map
\begin{equation}\label{e:integrand-continuous}
\tau\mapsto e^{A\tau}A^{-1}B\hat{u}(\tau,x)= -e^{A\tau}A^{-1}B\, B^*P\Phi(\tau)x
\end{equation}
is continuous on $[0,\infty)$ with values in $Y$, thereby ensuring that
\begin{equation*}
\exists \lim_{t\to 0^+}
\frac1{t}\int_0^t e^{A(t-\tau)}A^{-1}B\hat{u}(\tau,x)\,d\tau = -A^{-1}B\, B^*Px\,,
\quad x\in \cD(A_P)\,.
\end{equation*}

Now in view of the remarks above, \eqref{e:sum-of-limits} yields for any $x\in \cD(A_P)$  
\begin{equation*}
(A_Px,(A^*)^{-1}w)_Y=(x,w)_Y-(A^{-1}B\,B^*Px,w)_Y \qquad \forall w\in Y\,,
\end{equation*}
that is  
\begin{equation*}
(A^{-1}A_Px,w)_Y=(x-A^{-1}B\,B^*Px,w)_Y \qquad \forall w\in Y\,,
\end{equation*}
meaning that $x-A^{-1}B\,B^*Px\in \cD(A)$, with  
\begin{equation} \label{e:dom-a_p-described}
A_Px=A(x-A^{-1}B\,B^*Px) \qquad \forall x\in \cD(A_P)\,.
\end{equation} 
In addition, semigroup theory provides the basic differential property  
\begin{equation*}
\frac{d}{dt}\Phi(t)x= A[I-A^{-1}BB^*P]\Phi(t)x = \Phi(t)A[I-A^{-1}BB^*P]x 
\quad \forall x\in\cD(A_P)\,.
\end{equation*}

We have established Statement S5. of Theorem~\ref{t:main}, which is recorded in the following 
Proposition.
\begin{proposition}[Statement S5. of Theorem~\ref{t:main}] \label{p:statement-5}
The infinitesimal generator $A_P$ of the (optimal state) semigroup $\Phi(t)$ 
defined in \eqref{e:optimal-state-semigroup} coincides with the operator 
$A(I-A^{-1}BB^*P)$, on the domain
\begin{align}
\cD(A_P)& \subset \big\{x\in Y: x-A^{-1}BB^*Px\in \cD(A) \big\}
\label{e:inclusion_1}\\
& \subset \big\{ x\in Y: \exists \wlim_{t\to 0^+} \frac1{t} \int_0^t e^{A(t-\tau)}A^{-1}BB^*P\Phi(\tau)x\,d\tau 
\label{e:inclusion_2} 
\end{align}

\end{proposition}

\medskip

% THREE REMARKS

% Remark n. 1

\begin{remark} \label{r:operator-gamma}
\begin{rm}
We observe that for $x \in \cD(A_P)$ the weak limit 
\begin{equation*}
\lim_{t\to 0^+}\frac1{t}\Big(\int_0^t e^{A(t-\tau)}A^{-1}B\hat{u
}(\tau,x)\,d\tau,w\Big)_Y \quad \forall w\in Y
\end{equation*}
defines a linear operator, which we denote by $\Gamma$, that is 
\begin{align}
\Gamma x &:= \wlim_{t\to 0^+}\frac1{t}\int_0^t e^{A(t-\tau)}A^{-1}B\hat{u}(\tau,x)\,d\tau
\label{e:weak-limit_1}
\\
& =-\,\wlim_{t\to 0^+}\frac1{t}\int_0^t e^{A(t-\tau)}A^{-1}B\,B^*P\Phi(\tau)x\,d\tau\,,  
\label{e:weak-limit_2}
\end{align}
with dense domain in $Y$, as $\cD(A_P)\subseteq \cD(\Gamma)$.
Then, in the discussion leading to the statement of Proposition~\ref{p:statement-5}
we have shown that when $x\in \cD(A_P)$ then $\Gamma x=-A^{-1}B\,B^*Px$, 
i.e.~$\Gamma$ coincides with the operator $-A^{-1}B\,B^*P$ on $\cD(A_P)$.

We note, in addition, that if $x\in \cD(A^\epsilon)$, owing to \eqref{e:goal_2} of Proposition~\ref{p:statement-eight} the map in \eqref{e:integrand-continuous} is continuous as well, 
so that the weak limit \eqref{e:weak-limit_1} is strong, 
$\cD(A^\epsilon)\subseteq \cD(\Gamma)$, and we find again 
\begin{equation*}
\Gamma x= -A^{-1}B\,B^*Px \qquad x\in \cD(A^\epsilon)\,.
\end{equation*} 

\end{rm}
\end{remark}

% Remark n. 2

\begin{remark}
\begin{rm}
Pretty much in the same way we obtain for $x\in \cD(A)$ the existence of the weak limit
\begin{equation*}
\lim_{t\to 0^+}\frac1{t}\Big(\int_0^t A_P^{-1}e^{A(t-\tau)}B\hat{u}(\tau,x)\,d\tau,w\Big)_Y 
\quad \forall w\in Y\,, 
\end{equation*}
which defines a linear operator, which we denote by $\Gamma_P$, that is 
\begin{align*}
\Gamma_P x &:= \wlim_{t\to 0^+}\frac1{t}\int_0^t A_P^{-1}e^{A(t-\tau)}B\hat{u}(\tau,x)\,d\tau
%\label{e:weak-limit_3}
\\
& =-\,\wlim_{t\to 0^+}\frac1{t}\int_0^t A_P^{-1}e^{A(t-\tau)}B\,B^*P\Phi(\tau)x\,d\tau\,,  
%\label{e:weak-limit_2}
\end{align*}
with dense domain in $Y$, as $\cD(A)\subseteq \cD(\Gamma_P)$. 
Moreover, 
% similarly to the discussion leading to Proposition \ref{p:statement-5}, we have
we have 
\begin{equation*}
(Ax,{A_P^*}^{-1}w)_Y=(x,w)_Y-(\Gamma_Px,w)_Y \qquad \forall w\in Y\,,
\end{equation*}
i.e. $x-\Gamma_Px\in \cD(A_P)$, with  
\begin{equation*} %\label{e:dom-a-described}
Ax=A_P(x-\Gamma_Px) \qquad \forall x\in \cD(A)\,.
\end{equation*} 

It is interesting to note that although an explicit representation of $\Gamma_Px$ for 
$x\in \cD(A)$ is missing, we will later see that the key issue is not so much to be able to
deal with $\Gamma_Px$ by itself, but rather to find the ``right'' representation of 
the operator $e^{A_Pt}A_P\Gamma_P$, along with an appropriate regularity (in time and space).
Specifically, we will establish
\begin{equation*}
e^{A_Pt}A_P\Gamma_P\equiv e^{A_Pt}B\,B^*P: \cD(A) \longrightarrow [\cD({A^*}^\epsilon]' \qquad
\textrm{for a.e.~$t>0$,}
\end{equation*}
a property which is central to the proof of Corollary~\ref{c:differentiability-on-domain_A}
and then to well-posedness of the algebraic Riccati equations. 
\end{rm}
\end{remark}

% Remark n. 3

\begin{remark}
\begin{rm}
It is important to emphasize that the statement of Proposition~\ref{p:statement-5} cannot be improved to assert the equivalence between $\cD(A_P)$ and the subset of $Y$ in 
\eqref{e:inclusion_2} (which is characterized by the existence of the discussed weak 
limit $\Gamma x$).
In fact, assuming $x\in Y$ is such that $\Gamma x$ does exist, we only obtain that
\begin{equation*}
\exists \lim_{t\to 0^+ }\Big(\frac{\Phi(t)x-x}{t},z\Big) \qquad \forall z\in \cD(A^*)\,,
\end{equation*} 
which is not sufficient to conclude that $x\in \cD(A_P)$.

One may also wonder whether the inclusion in \eqref{e:inclusion_1} becomes an equality, possibly adding the restriction $x\in \cD(A^\epsilon)$.
However, this is not the case.
In fact, assuming $x\in \cD(A^\epsilon)$, with $x-A^{-1}BB^*Px\in \cD(A)$, from the decomposition 
\begin{align*} 
\Big(\frac{\Phi(t)x-x}{t},z\Big) &=  \Big(\frac{1}{t}\big[e^{At}-I\big](x - A^{-1}BB^*Px),z\Big) +
\\
& \quad +\Big(\frac{1}{t}\int_0^t e^{A(t-s)} BB^*P[x-\Phi(s)x],z\Big)\,,
\qquad t> 0\,, \; z\in Y\,,
\end{align*}
we immediately see that the first summand on the right hand side converges 
to $A(x - A^{-1}BB^*Px)$ as $t\to 0$, while to make sure that the second summand is infinitesimal 
we need $z\in \cD(A^*)$.
Thus, the same conclusion as above follows.
\end{rm}
\end{remark}

\bigskip
Although a full characterization of $\cD(A_P)$ is missing, the following Proposition
clarifies the relation between $\cD(A_P)$ and $\cD(A^\epsilon)$, as well as between
the domains of the corresponding adjoint operators.
This result will be also employed in the proof of Corollary~\ref{c:differentiability-on-domain_A}.
   
% Berlino
\begin{proposition} \label{p:berliner-inclusion}
The following inclusions are valid, provided $\epsilon< 1-\gamma$:
\begin{equation} \label{e:inclusion-of-domains}
\cD(A_P)\subseteq \cD(A^\epsilon)\,, 
\qquad
\cD(A^*_P)\subseteq \cD({A^*}^\epsilon)\,.
\end{equation}
  
\end{proposition}

\begin{proof}
We prove first the latter inclusion in \eqref{e:inclusion-of-domains}, which is in addition
of central importance in the proof of Corollary~\ref{c:differentiability-on-domain_A}; the 
former can be shown using similar arguments.
The proof is based on a well known characterization of the domains of fractional powers 
${A^*}^\epsilon$ in terms of the interpolation spaces $(Y,\cD(A^*))_{\epsilon,2}$.
The idea is to relate first $\cD(A^*_P)$ to one of the interpolation spaces 
$(Y,\cD(A^*))_{\alpha,\infty}$.
These can be also described as follows:  
\begin{equation}\label{e:caratterizzazione-spazio} 
(Y,\cD(A^*))_{\alpha,\infty}\equiv \big\{x\in Y: \, 
\sup_{t \in (0,1]}  t^{-\alpha}\|e^{A^*t}x -x\|_Y < \infty\big\}\,;
\end{equation}
see \cite[Theorem~1.13.2]{triebel}.

\smallskip
We aim to show that $\cD(A^*_P) \subset (Y,\cD(A^*))_{\alpha,\infty}$, 
for $\alpha\in (\epsilon,1)$. Let $x\in \cD(A^*_P)$. Then 
\begin{equation*}
[\Phi(t)^*-I]x=O(t)\,, \qquad t\to 0^+\,,
\end{equation*}
which implies 
\begin{equation*}
([\Phi(t)^*-I]x,z)_Y=O(t)\,, \qquad t\to 0^+\,, 
\end{equation*}
for all $z\in Y$. 
With $x\in \cD(A^*_P)$ fixed and any $z\in Y$, we rewrite 
\begin{equation}\label{e:somma}
\begin{split}
([\Phi(t)^*-I]x,z)_Y&= (x,[\Phi(t)-I]z)_Y 
\\
& =(x,[e^{At}-I]z)_Y + \underbrace{\big(x,\int_0^t e^{A(t-s)}B\hat{u}(s,z)\, ds\big)_Y}_{T_2(t)}
\end{split}
\end{equation}
and focus on the second summand $T_2$, which in turn splits as follows:
\begin{equation*}
T_2(t) =(x,L_{(1)}\hat{u}(t))_Y+(x,L_{(2)}\hat{u}(t))_Y\,.
\end{equation*}
(The operators $L_{(i)}$---resulting from the splitting of $B^*e^{A^*t}$---have been introduced in \eqref{e:input-to-state}; above, we set $L_{(i)}\hat{u}(t)$ in place of 
$(L_{(i)}\hat{u}(\cdot,z))(t)$, $i=1,2$, just for conciseness.)
It is readily seen that 
\begin{align}
|(x,L_{(1)}\hat{u}(t))_Y|
& \le \|x\|_Y\,\int_0^t \frac{C_1}{(t-s)^\gamma}\, \|\hat{u}(s,z)\|_U\, ds 
\nonumber \\ 
& 
\le C_1\,\|x\|_Y\, \|\hat{u}\|_{L^p(0,\infty;U)}\, 
\Big(\int_0^t\frac{1}{(t-s)^{\gamma p'}}\, ds\Big)^{1/p'}
\nonumber \\
& \le C_1\,t^{1/p'-\gamma}\, \|x\|_Y\, \|z\|_Y\,.
\label{e:pre-asymptotic_1}
\end{align}
To achieve the above estimate we have used Assumption~\ref{h:ip2}(i),
the H\"older inequality, as well as the continuity property \eqref{e:continuously-control} established in Proposition~\ref{p:reg-uy}.
Notice that by Proposition~\ref{p:reg-uy} $p$ can be taken arbitrarily large: in particular, here $p>1/(1-\gamma)$ is required, in order to ensure that the exponent $1/p'-\gamma$ is positive. 
In addition, as by assumption $1-\gamma-\epsilon>0$, then we may choose 
$p\ge (1-\gamma-\epsilon)^{-1}$ so that $1/p'-\gamma\ge \epsilon$, 
and \eqref{e:asymptotic_1} reads as 
\begin{equation} \label{e:asymptotic_1}
|(x,L_{(1)}\hat{u}(t))_Y| \le C_1\,t^{\alpha_1}\, \|x\|_Y\, \|z\|_Y
\end{equation} 
with $\alpha_1\ge \epsilon$.

As for the summand $(x,L_{(2)}\hat{u}(t))_Y$, by Proposition~\ref{p:stimeL2}(ii) 
we know that $t\mapsto L_{(2)}\hat{u}(t)$ is a continuous function on the whole half-line 
$[0,\infty)$, as $\hat{u}\in L^p(0,\infty;U)$ for any finite $p\ge 1$.
On the other hand, since we aim here to obtain an asympotic estimate of $L_{(2)}\hat{u}(t)$ 
as $t\to 0^+$, we may set $t\le T$.
Taking the inner product with any $x\in Y$, we find  
\begin{align}
& |(L_{(2)}\hat{u}(t),x)_Y| =\Big|\int_0^t \big(G(t-s)^*\hat{u}(s),x\big)_Y\,ds\Big|
\le \int_0^t \big|(\hat{u}(s),G(t-s)x)_U\big|\,ds
\label{e:before-holder} 
\\ 
& \qquad \le \|\hat{u}\|_{L^p(0,T;U)}\, \Big(\int_0^t \big\|G(t-s)x\big\|_U^q\,ds\Big)^{1/q}\,
\Big(\int_0^t 1\,ds\Big)^{1/r}
\label{e:after-holder}
\\
& \qquad \le \|\hat{u}\|_{L^p(0,\infty;U)}\,\|x\|_Y \,\|G(\cdot)\|_{\cL(Y,L^q(0,T;U))} \,t^{1/r}
\le C\,t^{1/r}\,\|z\|_Y\|x\|_Y \quad  \forall x\in Y\,,
\nonumber
\end{align}
where to go from \eqref{e:before-holder} to \eqref{e:after-holder} we applied the H\"older inequality with $1/p+1/q+1/r=1$, and by Assumption~\ref{h:ip2}(ii) the summability exponent 
$q$, like $p$, can be chosen freely as well. 
Notice that this makes it possible to render $1/r=1-1/p-1/q$ arbitrarily close to $1$.
The above computations yield the pointwise estimate 
% $\|L_{(2)}\hat{u}(t)\|_Y\le C_2\,t^{1/r}\,\|z\|_Y$,
% which gives 
\begin{equation} \label{e:asymptotic_2}
|(x,L_{(2)}\hat{u}(t))_Y|\le C_2\,t^{\alpha_2}\,\|x\|_Y\,\|z\|_Y
\end{equation}
with arbitrary $\alpha_2<1$. 
Thus, combining \eqref{e:asymptotic_1} with \eqref{e:asymptotic_2} we find that 
there exists a constant $C$ such that 
\begin{equation} \label{e:asymptotics-T_2}
\begin{split}
|T_2(t)| &\le |(x,L_{(1)}\hat{u}(t))_Y|+|(x,L_{(2)}\hat{u}(t))_Y|
\\
&\le C\,t^{\min\{\alpha_1,\alpha_2\}}\,\|x\|_Y\,\|z\|_Y= O(t^\alpha)\,\|x\|_Y\,\|z\|_Y\,,
\quad 0<t\le T\,,
\end{split}
\end{equation}
with $\epsilon<\alpha<1$.

Returning to \eqref{e:somma}, we have so far shown that 
% given $x\in \cD(A^*_P)$,
\begin{equation}
\|(e^{A^*t}-I)x\|_Y = O(t)-O(t^\alpha)= O(t^\alpha)\,, \qquad t\to 0^+\,,
\end{equation}
which in view of \eqref{e:caratterizzazione-spazio} establishes the membership 
$x\in (Y,\cD({A^*}))_{\alpha,\infty}$ for all $\alpha \in (\epsilon,1)$.

Thus, if we recall from \cite[Theorem~1.3.3]{triebel} the inclusions 
\begin{equation*}
(X,Y)_{\alpha,1} \subset (X,Y)_{\alpha,p} \subset (X,Y)_{\alpha,\infty} \subset (X,Y)_{\theta,1}\,,
\end{equation*}
which hold for all $\alpha$, $\theta$, $p$ such that $0<\alpha<\theta< 1$ 
and $1< p <\infty$, we immediately conclude that there exists $\theta\in (\epsilon,1)$ such that 
\begin{equation*}
x\in (Y,\cD({A^*}))_{\theta,2} \equiv\cD({A^*}^{\theta})\,;
\end{equation*}
see, e.g., \cite[\S~0]{las-trig-encyclopedia}. 
Consequently, $x\in \cD({A^*}^{\epsilon})$ which shows $\cD(A^*_P)\subset \cD({A^*}^{\epsilon})$,
thus concluding the proof.
\end{proof}

%-------------------------------------------------------------------------
\subsection{A distinct regularity result pertaining to $e^{A_Pt}$}
\label{ss:differentiate-Phi(t)B} 
%-------------------------------------------------------------------------
We discuss here a distinct regularity property of the optimal state semigroup 
$\Phi(t)=e^{tA_P}$ which will play a crucial role in the proof of well-posedness of the Algebraic Riccati Equations (ARE) corresponding to the optimal control problem.
It is indeed the regularity result established in Proposition~\ref{p:heritage} below which will make it possible to differentiate strongly the semigroup $e^{A_P t}$ on $\cD(A)$, as needed to obtain that the optimal cost operator $P$ does satisfy the ARE on $\cD(A)$.
% see Corollary~\ref{c:differentiability-on-domain_A}.

\begin{remark}
\begin{rm}
The same issue, that is strong differentiability of $e^{A_P t}$ on $\cD(A)$, was addressed as well in the study of the infinite horizon LQ-problem for abstract control systems  which yield {\em singular estimates}; see \cite[Section~3.3]{las-trig-se-1}.
In that work the sought property was easily established in view of the following crucial fact:
$\Phi(t)B$ inherited the same singular estimate as $e^{At}B$.
\end{rm}
\end{remark}

To begin with, let us preliminary state the intrinsic regularity of the map $t \mapsto e^{At}B$
for the class of control systems under investigation.

\begin{lemma} \label{l:reg-S(t)B}
Consider, for $t\ge 0$, the operator $e^{At}B$, defined in $U$ and taking values---a priori---in 
$[\cD(A^*)]'$.
For any $\delta \in [0,\omega\wedge \eta[$ we have
\begin{equation} \label{e:reg-S(t)B}
e^{\delta \cdot}e^{A\cdot}B\in \cL(U,L^s(0,\infty;[\cD({A^*}^\epsilon)]'))
\qquad \forall s\in [1,\frac1{\gamma})\,. 
\end{equation}
\end{lemma}

\begin{proof}
We use a duality argument.
With $u\in U$ and $f\in L^r(0,\infty;\cD({A^*}^\epsilon))$, $1/(1-\gamma)< r\le \infty$,
we estimate
\begin{eqnarray}
\lefteqn{\hspace{-15mm}
\Big|\int_0^\infty
(e^{\delta t}e^{At}Bu,f(t))_{\cD({A^*}^\epsilon)}\,dt\Big|
=\Big|\int_0^\infty(e^{-\theta t}u,e^{(\delta+\theta)t} B^*e^{A^*t}f(t))_U\,dt\Big|
}
\nonumber \\
&\le & \Big(\int_0^\infty e^{-\theta p't}\, dt\Big)^{1/p'}\,\|u\|_U 
\|e^{(\delta+\theta)\cdot} B^*e^{A^*\cdot}f(\cdot)\|_{L^p(0,\infty;U)}
\label{e:pre-by-duality}\\
& \le & C\,\|u\|_U\, \|f\|_{L^r(0,\infty;\cD({A^*}^\epsilon))}\,,
\label{e:by-duality}
\end{eqnarray}
where $\theta$ is any positive number such that $\delta+\theta<\omega \wedge \eta$, 
$p\in [1,1/\gamma)$ is chosen in order to fulfil the bounds
\begin{equation*}
\frac1{1-\gamma}\le \frac{p}{1-\gamma p}< r
\end{equation*}
($p'$ is its conjugate exponent), and we utilized Proposition~\ref{p:expGf} to go from 
\eqref{e:pre-by-duality} to \eqref{e:by-duality}.
Notice that Proposition~\ref{p:expGf} applies, since the required constraint 
\eqref{e:constraint-r} is satisfied.

Thus, \eqref{e:by-duality} shows that the map $t \mapsto e^{\delta t}e^{At}Bu$ belongs to 
$L^{r'}(0,\infty;[\cD({A^*}^\epsilon)]')$, $r'$ being the conjugate exponent of $r$.
From $r\in (1/(1-\gamma),\infty]$ we get $r'\in [1, 1/\gamma)$, and 
\eqref{e:reg-S(t)B} holds true with $s=r'\in [1,1/\gamma)$, as desired.

\smallskip
Notice carefully that the range of the summability exponent $s$ for the validity 
of \eqref{e:reg-S(t)B} cannot be improved. 
In fact, owing to Proposition~\ref{p:expGf} the exponent $r$ in the obtained estimate 
\eqref{e:by-duality} is subject to the constraint \eqref{e:constraint-r},
which implies 
\begin{equation*}
1\le r'< \frac{p}{p-(1-\gamma p)}\,,
\end{equation*}
while it is readily verified that 
\begin{equation*}
\sup_{1< p< \frac1{\gamma}}\frac{p}{p-(1-\gamma p)}= \frac1{\gamma}\,,
\end{equation*}
This confirms that \eqref{e:reg-S(t)B} holds true if and only if $s\in [1,1/\gamma)$,
thus concluding the proof.
\end{proof}

\smallskip
In order to pinpoint the regularity of $\Phi(t)B$, we will employ the usual
representation of the optimal state in terms of the initial state.
It follows from \eqref{e:introduces-Lambda} that
\begin{equation*} 
\Phi(\cdot)x = (I+LL^*R^*R)^{-1}e^{A\cdot}x\,,
\end{equation*}
which becomes 
\begin{equation} \label{e:phi}
\Phi(\cdot)x = \big(\,I-L\Lambda^{-1}L^*R^*R\,\big)\,e^{A\cdot}x\,,
\end{equation}
where $\Lambda=I+L^*R^*RL$ is boundedly invertible on $L^2(0,\infty;U)$---an elementary consequence of the fact that $\Lambda$ is {\em coercive} on $L^2(0,\infty;U)$.

\begin{remark}
\begin{rm}
We note that the representation of the inverse $(I+LL^*R^*R)^{-1}$ which occurs in 
the formula \eqref{e:phi} is easily derived by a direct (algebraic) computation. 
The inversion of an operator of the form $I+SV$ in a Hilbert space setting is discussed 
in full detail in \cite[Lemma~2A.1, p.~167]{las-trig-encyclopedia}.
\end{rm}
\end{remark}

If we take now $x=Bu$ in \eqref{e:phi} and formally rewrite the corresponding formula,
we obtain the following representation for $\Phi(t)Bu$:
\begin{equation} \label{e:phi-b}
e^{A_Pt}Bu\equiv \Phi(t)Bu = e^{At}Bu - [L\Lambda^{-1}L^*R^*R e^{A\cdot}Bu](t)\,,
\end{equation}
where $\Lambda^{-1}$ is required to make sense on the space $L^q(0,\infty;U)$,
or $L_{\delta}^q(0,\infty;U)$ for some positive $\delta$, rather than on $L^2(0,\infty;U)$.
Indeed, given $u\in U$, by Lemma~\ref{l:reg-S(t)B} we know that 
\begin{equation*}
e^{A\cdot}Bu\in L_{\delta}^1(0,\infty;[\cD({A^*}^\epsilon)]')\,.
\end{equation*} 
Then, owing to Proposition~\ref{p:L*R*R}, the application of the operator $L^*R^*R^*$
yields 
\begin{equation*}
[L^*R^*R^*e^{A\cdot}Bu](t)\in L_{\delta}^q(0,\infty;U)\,,
\end{equation*}
which holds for any $\delta \in [0,\omega\wedge \eta[$.

The question which then arises is the following.
\begin{question}
Is the operator $\Lambda= I+L^*R^*RL$ boundedly invertible on the function space 
$L_{\delta}^q(0,\infty;U)$?
\end{question}
It will become clear in the proof of Theorem~\ref{t:invertible-on-X_q} below
that in contrast with previous theories $\Lambda$ is {\em not} an isomorphism
on $L_{\delta}^q(0,\infty;U)$ (with a fixed $\delta$), as we would expect. 
An intermediate useful result is the one given in the following Lemma.

\begin{lemma} \label{l:invertible-on-L2delta}
The operator $\Lambda = I + L^*R^*RL$ is an isomorphism in the space 
$L_{\delta}^2(0,\infty;U)$, provided that $\delta\in (0, \omega\wedge \eta)$
is sufficiently small.
\end{lemma}

\begin{proof}
We seek to solve uniquely the equation
\beq\label{e:base}
w + L^*R^*RLw = h,\qquad h\in L_{\delta}^2(0,\infty;U).
\eeq
Let us denote by $H$ the function space $L_{\delta}^2(0,\infty;U)$.
Since $H\subseteq L^2(0,\infty;U)$, there is a unique $w\in L^2(0,\infty;U)$ such that \eqref{e:base} holds. 
Multiplying \eqref{e:base} by $e^{\delta t}$, we get 
\begin{equation*}
e^{\delta t} w + e^{\delta t} L^*R^*RLw = e^{\delta t}h\,,
\end{equation*}
which is equivalent to 
\begin{equation}\label{e:trasf}
e^{\delta t} w + L^*_{A-\delta}R^*RL_{A+\delta} (e^{\delta t}w)= e^{\delta t}h\,,
\end{equation}
where we denoted by $L_{A+\delta}$ the input-to-state map for the control system
$y'=(A+\delta)y+Bv$, $y(0)=0$, namely 
\begin{equation*}
L_{A+\delta}v(t) := \int_0^t e^{(A+\delta)(t-s)}Bv(s)\,ds\,;
\end{equation*}
$L^*_{A-\delta}$ is defined accordingly.
Thus, in order to simplify the notation, let us rewrite \eqref{e:trasf} as follows:
\beq\label{e:trasf-prime}
e^{\delta t} w + L^*_{-\delta}R^*RL_{\delta} (e^{\delta t}w)= e^{\delta t}h\in L^2(0,\infty;U)\,.
\eeq
We now utilize the estimate
\begin{equation} \label{e:smallness}
\|L_{-\delta}^*R^*RL_\delta - L^*R^*RL\|_{\cL(L^2(0,\infty;U))} 
\le c\,\frac{\delta}{(\eta\wedge \omega - \delta)^2}\,,
\end{equation}
which will be established in Lemma~\ref{l:smallness} below.
The above implies that $I+L^*_{-\delta}R^*RL_\delta$ is also invertible in $L^2(0,\infty;U)$, with continuous inverse, for sufficiently small $\delta$.
Hence, the equation 
\begin{equation*}
z + L^*_{-\delta}R^*RL_\delta z= e^{\delta t}h
\end{equation*}
has a unique solution $z\in L^2(0,\infty; U)$.
Observe now that the function $e^{-\delta t}z$ belongs to $H\subset L^2(0,\infty;U)$ and satisfies 
\begin{equation} 
e^{-\delta t}z + L^*R^*RL (e^{-\delta t}z)= h\,.
\end{equation}
Comparing \eqref{e:base} with the above gives, by uniqueness, $e^{-\delta t}z\equiv w$, 
so that $w\in H$. 
This shows that \eqref{e:base} is uniquely solvable in $H$.
The proof is completed once we establish the estimate \eqref{e:smallness}.
This is accomplished in the following Lemma.
\end{proof}

\begin{lemma} \label{l:smallness}
If $\delta \in \,]0,\eta\wedge \omega[\,$, the estimate \eqref{e:smallness} holds true.
\end{lemma}
\begin{proof}
We have
$$L^*_{-\delta}R^*RL_\delta - L^*R^*RL = [L^*_{-\delta}-L^*]R^*RL_\delta 
+ L^*R^*R[L_\delta-L].$$
Consider now $L_\delta-L$. It holds for each $u\in L^2(0,\infty;U)$
\beyy
[L_\delta-L]y(t) & = & \int_0^t e^{A(t-s)}Bu(s) [e^{\delta(t-s)}-1]ds = \\
& = & \int_0^t e^{\delta(t-s)} e^{A(t-s)}Bu(s) [1-e^{-\delta(t-s)}]ds,
\eeyy
so that we can split
\beyy
[L_\delta-L]y(t) & = & \int_0^t e^{\delta(t-s)}F(t-s)^*u(s)[1-e^{-\delta(t-s)}]ds + \\
& + & \int_0^t e^{\delta(t-s)}G(t-s)^*u(s)[1-e^{-\delta(t-s)}]ds.
\eeyy
The first term can be estimated by
\beyy
\lefteqn{\left\|\int_0^t e^{\delta(t-s)}F(t-s)^*u(s)[1-e^{-\delta(t-s)}]ds\right\|_Y \le}\\
& & \le c\delta \int_0^t (t-s)^{1-\gamma}e^{-(\eta-\delta)(t-s)} \|u(s)\|_U \,ds,
\eeyy 
and it is straightforward to deduce that
\beyy
\lefteqn{\left[\int_0^\infty \left\|\int_0^t e^{\delta(t-s)}F(t-s)^*u(s)[1-e^{-\delta(t-s)}]ds\right\|_Y^2 dt\right]^{\frac1{p}} \le}\\
& & \le c\delta \int_0^t \sigma^{1-\gamma}e^{-(\eta-\delta)\sigma} d\sigma \|u\|_{L^2(0,\infty;U)}\le C \frac{\delta}{(\eta-\delta)^{2-\gamma}}\|u\|_{L^2(0,\infty;U)}.
\eeyy
The second term is estimated as follows: fix $\psi \in L^2(0,\infty;U)$ and set $\alpha= \frac{\delta + \eta \wedge \omega}{2}$, $\beta= \frac{3\delta + \eta \wedge \omega}{4}= \frac{\delta + \alpha}{2}$. Then we have
\beyy
\lefteqn{\int_0^\infty\left( \int_0^t e^{\delta(t-s)}G(t-s)^*u(s)[1-e^{-\delta(t-s)}]ds, \psi(t)\right)_Y \,dt =}\\
& & = \int_0^\infty\left(\int_0^t u(s), e^{\delta(t-s)}[1-e^{-\delta(t-s)}]G(t-s)\psi(t)\right)_Y \,dsdt = \\
& & = \int_0^\infty \hspace{-2mm}\int_0^t e^{-(\beta-\delta)(t-s)}\|u(s)\|_U \,\|e^{\beta(t-s)}[1-e^{-\delta(t-s)}]\|G(t-s)\psi(t)\|_U \,dsdt \le \\
& & \le c\delta \int_0^\infty \|e^{-(\beta-\delta)(t-\cdot)}u\|_{L^2(0,t;U)}\|(t-\cdot)\,e^{\beta(t-\cdot)}G(t-\cdot)\psi(t)\|_{L^2(0,t;U)} dt \le \\
& & \le c\frac{2\delta}{\eta\wedge \omega-\delta} \int_0^\infty \|e^{-(\beta-\delta)(t-\cdot)}u\|_{L^2(0,t;U)}\|e^{\alpha \cdot}G(\cdot)\psi(t)\|_{L^2(0,t;U)} dt \le\\
& & \le c\frac{2\delta}{\eta\wedge \omega-\delta}\int_0^\infty \|e^{-(\beta-\delta)(t-\cdot)}u\|_{L^2(0,t;U)}\|\psi(t)\|_Y \, dt.
\eeyy
From here it is a standard matter to deduce that
\beyy
\lefteqn{\int_0^\infty\left( \int_0^t e^{\delta(t-s)}G(t-s)^*u(s)[1-e^{-\delta(t-s)}]ds, \psi(t)\right)_Y \,dt \le}\\
& & \le C\frac{\delta}{(\eta\wedge \omega-\delta)^2}\|u\|_{L^2(0,\infty;U)}\|\psi\|_{L^2(0,\infty;U)}\,.
\eeyy
This shows that 
\beyy
\lefteqn{\left\|\int_0^t e^{\delta(t-s)}G(t-s)^*u(s)[1-e^{-\delta(t-s)}]ds\right\|_{L^2(0,\infty;Y)}\le} 
\\
& & \le C\frac{\delta}{(\eta\wedge \omega-\delta)^2}\|u\|_{L^2(0,\infty;U)} \,,
\eeyy
and summing up we obtain
\beq\label{stimaL}
\|[L_\delta-L]u\|_{L^2(0,\infty;Y)} \le C\frac{\delta}{(\eta\wedge \omega-\delta)^2}\|u\|_{L^2(0,\infty;U)} \,.
\eeq
Next, consider $L^*_{-\delta} -L^*$. It holds for each $y\in L^2(0,\infty;Y)$
\beyy
\lefteqn{[L^*_{-\delta} -L^*]y(t) = \int_t^\infty B^* e^{A^*(\tau -t)} [e^{-\delta(\tau-t)} -1] y(\tau)d\tau =}\\
& & =\int_t^\infty \hspace{-2mm}F(\tau-t)[e^{-\delta(\tau-t)} -1] y(\tau)d\tau + \int_t^\infty\hspace{-2mm} G(\tau-t)[e^{-\delta(\tau-t)} -1] y(\tau)d\tau.
\eeyy
From here, proceeding quite similarly to the preceding case, we get
\beq\label{stimaL*}
\|[L^*_{-\delta}-L^*]y\|_{L^2(0,\infty;U)} \le C\frac{\delta}{(\eta\wedge \omega-\delta)^2}\|y\|_{L^2(0,\infty;y)} \,.
\eeq
Finally, we can write
$$[L^*_{-\delta}R^*RL_\delta - L^*R^*RL]u = [L^*_{-\delta}-L^*]R^*RL_\delta u + L^*R^*R[L_\delta-L]u$$
and both terms can be easily estimated by \eqref{stimaL} and \eqref{stimaL*}. The result follows. 
\end{proof}

\medskip
We now utilize Lemma~\ref{l:invertible-on-L2delta} to show that the operator $\Lambda$
admits a bounded inverse $\Lambda^{-1}$ which maps $L^q_\delta(0,\infty;U)$ onto
$L^q_{\delta-\sigma_0}(0,\infty;U)$, for a suitable $\sigma_0\in (0,\delta)$.
% in \eqref{e:phi-b} 

\begin{theorem} \label{t:invertible-on-X_q}
The operator $\Lambda = I + L^*R^*RL$ admits a bounded inverse 
\begin{equation*}
\Lambda^{-1}: L^q_\delta(0,\infty;U) \longrightarrow L^q_{\delta-\sigma_0}(0,\infty;U)\,,
\end{equation*}
with appropriate $\sigma_0\in (0,\delta)$.
\end{theorem}

\begin{proof}
We seek to solve uniquely the equation
\begin{equation} \label{e:goal}
g + L^*R^*RLg = h\,,
\end{equation}
where $h\in L^q_\delta(0,\infty;U)$, with arbitrary $\delta\in (0,\omega\wedge\eta)$.
We follow an idea which has been employed in the study of the LQ-problem for parabolic-like dynamics; see, e.g., \cite[Vol.~I,~Theorem~1.4.4.4,~p.~40]{las-trig-encyclopedia}.
Since $q<2$, using the action of both the operators $L$ and $L^*$, whose mapping 
increase the (time regularity) summability exponents, it is readily seen that there exists an integer $n_0\ge 1$ such that $(L^*R^*RL)^{n_0}h\in L_{\delta}^2(0,\infty;U)$.
Thus, we introduce the auxiliary equation 
\begin{equation} \label{e:auxiliary}
v + L^*R^*RLv = (L^*R^*RL)^{n_0}h\in L_{\delta}^2(0,\infty;U)\,.
\end{equation}
Owing to Lemma~\ref{l:invertible-on-L2delta}, possibly choosing $\delta$ sufficiently small, \eqref{e:auxiliary} is uniquely solvable, yielding $v\in L_{\delta}^2(0,\infty;U)$.
Using once again that $q<2$, it is easily verified that the obtained $v$ belongs to 
$L_{\theta}^q(0,\infty;U)$ for any $\theta<\delta$. 
In fact,
\begin{align}
\int_0^\infty e^{\theta qt} \|v(t)\|_U^q\, dt 
&=\int_0^\infty e^{-(\delta-\theta)qt}\,\|e^{\delta t} v(t)\|_U^q\, dt
\label{e:elementary-split}\\
& \le \Big(\int_0^\infty e^{-[2(\delta-\theta)q/(2-q)]t}\,dt\Big)^{(2-q)/2} 
\Big(\int_0^\infty \|e^{\delta t} v(t)\|_U^2\, dt\Big)^{q/2}\,,
\nonumber
\end{align}
and there exists a constant $C$ such that 
\begin{equation*} 
\|v\|_{L_{\theta}^q(0,\infty;U)}\le C\,\|v\|_{L_{\delta}^2(0,\infty;U)}\,.
\end{equation*}
The function $v$ resulting from \eqref{e:auxiliary} and the given $h\in W$---for which the 
smaller $\theta$ still guarantees $e^{\theta\cdot}h\in L^q(0,\infty;U)$---will eventually 
produce the soughtafter solution $g$ of equation \eqref{e:goal}, 
according to the following definition:
$$g=\sum_{j=0}^{n_0-1}(-L^*R^*RL)^j h + v.$$
Indeed we have:
\beyy
\lefteqn{(I + L^*R^*RL)g =}\\
& & =\sum_{j=0}^{n_0-1}(-L^*R^*RL)^j h + \sum_{j=0}^{n_0-1}L^*R^*RL(-L^*R^*RL)^j h + v + L^*R^*RLv = \\
& & =\sum_{j=0}^{n_0-1}(-L^*R^*RL)^j h - \sum_{j=0}^{n_0-1}(-L^*R^*RL)^{j+1} h + (-L^*R^*RL)^{n_0}h = \\
& & = \sum_{j=0}^{n_0-1}(-L^*R^*RL)^j h - \sum_{i=1}^{n_0}(-L^*R^*RL)^i h + (-L^*R^*RL)^{n_0}h = \\
& & = h -(-L^*R^*RL)^{n_0}h + (-L^*R^*RL)^{n_0}h = h\,,
\eeyy 
which concludes the proof.
\end{proof}

\smallskip
We are finally able to show that the operator $e^{A_Pt}B$ substantially `inherits' the regularity 
of $e^{At}B$, except for a constraint on the exponent of the allowed exponential weights.

\begin{proposition}[Statement~S6.~of Theorem~\ref{t:main}] \label{p:heritage}
For $t\ge 0$, the linear operator $e^{A_Pt}B$ is well defined as an operator from 
$U$ into $\cD({A^*}^\epsilon)]'$ and, in fact, provided $\delta \in (0,\omega\wedge \eta)$ is sufficiently small, we have 
\begin{equation} \label{e:inherited}
e^{\delta \cdot}e^{A_P\cdot}B\in \cL(U,L^p(0,\infty;[\cD({A^*}^\epsilon)]'))
\qquad \forall p\in [1,\frac1{\gamma})\,. 
\end{equation}
\end{proposition}

\begin{proof}
Let $u\in U$ be given.
We return to the representation \eqref{e:phi-b} for $e^{A_Pt}Bu$, and focus on its second summand.
Starting from $e^{At}Bu$, whose regularity is established in Lemma~\ref{l:reg-S(t)B},  
we utilize \eqref{e:crucial} of Proposition~\ref{p:L*R*R} first, and invoke
Theorem~\ref{t:invertible-on-X_q} next, thus obtaining---possibly choosing 
$\theta<\omega\wedge \eta$ sufficiently small---,
\begin{equation*}
\Lambda^{-1}L^*R^*R e^{A\cdot}Bu\in L_{\theta}^q(0,\infty;U)\,.
\end{equation*}
In particular, there exists $\beta<\theta$ such that 
\begin{equation*}
\Lambda^{-1}L^*R^*R e^{A\cdot}Bu\in L_{\beta}^1(0,\infty;U)\,;
\end{equation*}
consequently, Proposition~\ref{p:stimeL}(i) implies 
\begin{equation}  \label{e:pezzo}
L\Lambda^{-1}L^*R^*R e^{A\cdot}Bu\in L_{\beta}^r(0,\infty;[\cD({A^*}^\epsilon)]')
\quad \forall r\in [1,\frac1{\gamma})\,.
\end{equation}
Actually, the range of the summability exponent in \eqref{e:pezzo} is larger; namely, 
the membership in \eqref{e:pezzo} holds true for all $r$ in a suitable maximal interval 
$I\supset [1,1/\gamma)$, which is determined by the reciprocal relation between 
$1/(1-\gamma)$ and $q$; see the statements (ii)--(iv) of Proposition~\ref{p:stimeL}. 
Even an `improved version' of the regularity in \eqref{e:pezzo}, combined with the one 
in \eqref{e:reg-S(t)B}, yields anyhow
\begin{equation*}
e^{A_P\cdot}Bu=e^{A\cdot}Bu-[L\Lambda^{-1}L^*R^*R e^{A\cdot}Bu](\cdot)
\in L_{\beta}^p(0,\infty;[\cD({A^*}^\epsilon)]')
\quad \forall p\in [1,\frac1{\gamma})\,,
\end{equation*} 
for suitably small $\beta\in (0,\omega\wedge \eta)$, confirming \eqref{e:inherited}.
\end{proof}

\smallskip
The power of the (apparently weak) regularity result provided by Proposition~\ref{p:heritage} is enlightened in the following Corollary. 

\begin{corollary}\label{c:differentiability-on-domain_A}
The optimal state semigroup $\Phi(t)=e^{A_Pt}$ is strongly differentiable on $\cD(A)$: 
more precisely, if $x\in \cD(A)$ the map $t\mapsto \Phi(t)x=e^{A_Pt}x$ is strongly 
differentiable for almost any $t>0$, with 
\begin{equation} \label{e:formula-for-derivative}
\frac{d}{dt}e^{A_Pt}x=e^{A_Pt}Ax-e^{A_Pt}B\,B^*Px\,, 
\quad \textrm{for $x\in \cD(A)$ and a.e.~$t>0$,}
\end{equation}
and the equality holds true on $[\cD({A^*}^\epsilon)]'$.
In particular, the operator 
\begin{equation} \label{e:regularity-of-derivative}
\frac{d}{dt}e^{A_Pt}\; \textrm{is continuous:} \quad
\cD(A) \longrightarrow L_{\delta}^p(0,\infty;[\cD({A^*}^\epsilon)]')
\qquad \forall p\in [1,\frac1{\gamma})
\end{equation}
provided $\delta \in (0,\omega\wedge \eta\wedge \omega_1)$ is sufficiently small,
and the following estimate holds true almost everywhere in $(0,\infty)$:
\begin{equation} \label{e:corresponding-estimate}
\begin{aligned} 
\Big\|A^{-\epsilon}\frac{d}{dt}e^{A_Pt}\Big\|_Y 
& \le M_1 e^{-\omega_1t}\|x\|_{\cD(A)} 
\\[1mm]
& \quad + \big\|A^{-\epsilon}e^{A_Pt}B\big\|_{\cL(U,Y)}\,
\|B^*P\|_{\cL(\cD(A^\epsilon),U)}\,\|x\|_{\cD(A^\epsilon)}\,.
\end{aligned}
\end{equation}

\end{corollary}

\begin{proof}
Let $x\in \cD(A)$. We preliminary note that $e^{A_Pt}x$ is strongly differentiable
as an element of the dual space $[\cD(A^*_P)]'$.
In fact, if $z\in \cD(A^*_P)$, then $z={A_P^*}^{-1}w$ with $w\in Y$ and we may compute
\begin{align*}  
& \frac{d}{dt}(e^{A_Pt}x,z)_{[\cD(A^*_P)]',\cD(A^*_P)}
=\frac{d}{dt}(A_P^{-1}e^{A_Pt}x,w)_Y = (e^{A_Pt}x,w)_Y
\nonumber \\
& \qquad =(e^{A_Pt}x,A_P^*z)_Y = (A_Pe^{A_Pt}x,z)_{[\cD(A^*_P)]',\cD(A^*_P)}
=(x,e^{A_P^*t}A^*_Pz)_Y
\nonumber \\
& \qquad
=(x,A^*_Pe^{A_P^*t}z)_Y=(e^{A_Pt}A_Px,z)_{[\cD(A^*_P)]',\cD(A^*_P)}
\nonumber \\
& \qquad = (e^{A_Pt}A[I-A^{-1}B\,B^*P]x,z)_{[\cD(A^*_P)]',\cD(A^*_P)}\,,
% \label{e:for-differentiability_0} 
\end{align*}
which shows that when $x\in \cD(A)$
\begin{equation} \label{e:preliminary-formula-for-derivative}
\exists \;\frac{d}{dt}e^{A_Pt}x=e^{A_Pt}A[I-A^{-1}B\,B^*P]x\,, 
\quad \textrm{for a.e.~$t>0$,}
\end{equation}
as an element of $[\cD(A^*_P)]'$.
In addition, since $e^{A_Pt}Ax\in Y$, \eqref{e:preliminary-formula-for-derivative}
yields---still for $x\in \cD(A)$ and a.e.~$t>0$---, 
\begin{equation*} 
e^{A_Pt}A\,[A^{-1}B\,B^*P]x=e^{A_Pt}Ax-e^{A_Pt}A[I-A^{-1}B\,B^*P]x\in [\cD(A^*_P)]'\,.
\end{equation*}

Next, we observe that $x\in \cD(A)$ implies $x\in \cD(A^\epsilon)$ for all $\epsilon\in (0,1)$, which ensures $B^*Px\in U$ in view of Theorem~\ref{t:bounded-gain}.
Then $BB^*Px\in [\cD(A^*)]'$ and also $A\,A^{-1}BB^*Px\in [\cD(A^*)]'$, because 
$AA^{-1}$ coincides with the identity operator on $[\cD(A^*)]'$.
We now recall \eqref{e:inherited} from Proposition~\ref{p:heritage} which establishes
$e^{A_Pt}B\,B^*Px[\cD({A^*}^\epsilon)]'$, along with the regularity (in time) result
\begin{equation} \label{e:membership-1}
e^{\delta\cdot}e^{A_P\cdot}B\,B^*Px\in L^p(0,\infty;[\cD({A^*}^\epsilon)]')
\qquad \forall p\in [1,\frac1{\gamma})\,,
\end{equation}
valid for sufficiently small $\delta>0$.
On the other hand, since the semigroup $e^{A_Pt}$ is exponentially stable (with the estimate
\eqref{e:exponential_1}, we find
\begin{equation} \label{e:membership-2}
e^{\delta\cdot}e^{A_P\cdot}Ax\in L^s(0,\infty;Y) \qquad \forall s\in [1,\infty]\,,
\end{equation}
for any $\delta<\omega_1$.
In view of the memberships \eqref{e:membership-1} and \eqref{e:membership-2} we see that 
\begin{equation} \label{membership_combined}
e^{A_P\cdot}A[I-A^{-1}B\,B^*P]x=e^{A_P\cdot}Ax-e^{A_P\cdot}B\,B^*Px
\in L^p_{\delta}(0,\infty;[\cD({A^*}^\epsilon)]')\,,
\end{equation}
provided $\delta$ is sufficiently small.

Thus, \eqref{membership_combined} shows that the derivative in \eqref{e:preliminary-formula-for-derivative}---{\em a priori} taking values on $[\cD(A^*_P)]'$---coincides with 
$e^{A_Pt}Ax-e^{A_Pt}A^{-1}B\,B^*Px\in [\cD({A^*}^\epsilon)]'$ for a.e. $t>0$
(We recall that the inclusion $[\cD(A^*_P)]'\supset [\cD({A^*}^\epsilon)]'$---which
holds true provided $\epsilon<1-\gamma$---is the dual statement of \eqref{e:inclusion-of-domains}
of Proposition~\ref{p:berliner-inclusion}, whereas here $\epsilon$ can be taken arbitrarily small.)
Therefore, \eqref{e:formula-for-derivative} actually makes sense on $[\cD({A^*}^\epsilon)]'$ 
for a.e. $t>0$, and the validity of \eqref{e:regularity-of-derivative} 
is established, provided $\delta$ is sufficiently small.

Finally, natural estimates for each summand in the right hand side of 
\eqref{e:formula-for-derivative} produce the bound in \eqref{e:corresponding-estimate}, 
thus completing the proof.
\end{proof}

%-----------------------------------------------------------------------------
\subsection{Well-posedness of the ARE}
%-----------------------------------------------------------------------------
We begin with the statement of a Lemma which provides boundedness of the 
operators $A^*P$ and $A_P^*P$ on appropriate spaces ($\cD(A)$ and $\cD(A_P)$, 
respectively), properties which constitute a prerequisite for well-posedness of the ARE.
Although the proof is fairly standard, it is given below for the reader's convenience.

\begin{lemma} \label{l:riccati-op}
The following statements pertain to the optimal cost operator $P$.

\begin{enumerate}
\item[(i)]
$A^*P\in \cL(\cD(A_P),Y)$, with
\begin{equation} \label{e:identity-1}
A^*Px=-R^*R x-PA_Px \qquad \forall x\in \cD(A_P)\,;
\end{equation}

\item[(ii)] 
$A_P^*P\in \cL(\cD(A),Y)$, with
\begin{equation} \label{e:identity-2}
A_P^*Px=-R^*R x-PAx \qquad \forall x\in \cD(A)\,.
\end{equation}

\end{enumerate}

\begin{proof}
(i) Let $x\in \cD(A_P)$. Write the formula \eqref{e:optimal-cost-op} which defines the 
Riccati operator, that is 
\begin{equation*}
Px = \int_0^\infty e^{A^*t}R^*R \Phi(t) x\,dt\,,
\end{equation*}
and integrate by parts in $t$, thus obtaining
\begin{equation*}
\begin{aligned}
Px &= \int_0^\infty {A^*}^{-1}A^*e^{A^*t}R^*R \Phi(t) x\, dt 
\\[1mm]
 & = {A^*}^{-1}e^{A^*t}R^*R \Phi(t) x\Big|_{t=0}^{t=\infty} 
- \int_0^\infty {A^*}^{-1} e^{A^*t}R^*R \Phi(t) A_P x\,dt
\\[1mm]
 & = -{A^*}^{-1}R^*R x - {A^*}^{-1}\int_0^\infty e^{A^*t}R^*R \Phi(t) A_P x\,dt
\\[1mm]
 & = -{A^*}^{-1}R^*R x - {A^*}^{-1}PA_P x\,.
\end{aligned}
\end{equation*}
The above identity shows that $P$ maps $\cD(A_P)$ into $\cD(A^*)$, and also that 
\eqref{e:identity-1} holds actually in $Y$, 
since $R^*R x - PA_P x\in Y$ for any $x\in \cD(A_P)$; the boundedness of $A^*P$ immediately follows.   

\smallskip
\noindent
(ii) We write 
\beyy
\lefteqn{(Px,z)_Y = \Big(\int_0^\infty e^{A^*t}R^*R e^{A_P t} x\,dt,z\Big)_Y 
=\int_0^\infty \big(e^{A^*t}R^*R e^{A_P t} x,z\big)_Y\, dt}
\\[1mm]
& & = \int_0^\infty \big(x,e^{{A_P}^*t}R^*R e^{At}z\big)_Y\, dt 
= \Big(x,\int_0^\infty e^{{A_P}^*t}R^*R e^{At}z\,dt \Big)_Y = (x,P^*z)_Y\,.
\eeyy
Since we know that $P=P^*$, we deduce the alternative formula
\begin{equation} \label{e:alternative-formula}
Px=\int_0^\infty e^{A_P^* t}R^*R e^{At}x\,dt\,, \qquad x\in Y\,. 
\end{equation}
Thus, the proof of \eqref{e:identity-2} follows almost precisely as in the proof
of \eqref{e:identity-1}, bringing about the equality
\begin{equation*}
Px= -{A_P^*}^{-1}R^*R x - {A_P^*}^{-1}PA x\, \qquad x\in \cD(A)\,,
\end{equation*}
which confirms that $A_P^*P$ is a bounded operator on $\cD(A)$ with values
in $Y$, and hence \eqref{e:identity-2} holds in $Y$.

\end{proof}
\end{lemma}

We are finally ready to show that the optimal cost operator is a solution of the the Algebraic Riccati equation corresponding to Problem~\ref{p:problem-0}.

\begin{theorem}[Statement~S7.~of Theorem~\ref{t:main}] \label{t:are-wellposed} 
The optimal cost operator $P$ defined in \eqref{e:optimal-cost-op} 
satisfies the following regularity properties:
\begin{equation*}
P\in \cL(\cD(A_P),\cD(A^*))\cap \cL(\cD(A),\cD(A^*_P)\,.
\end{equation*}
Moreover, $P$ is a solution to the Algebraic Riccati equation 
\begin{equation}
\begin{split}
& (Px,Az)_Y+(Ax,Pz)_Y-(B^*Px,B^*Pz)_U+(Rx,Rz)_Z=0 
\label{e:riccati-eq} \\[1mm]
& \myspace\myspace\myspace\textrm{for any $x,z\in \cD(A)$,}
\end{split}
\end{equation}
which reads as 
\begin{equation}
\begin{split}
& (A^*Px,z)_Y+(x,A^*Pz)_Y-(B^*Px,B^*Pz)_U+(Rx,Rz)_Z=0 
\label{e:riccati-eq_2} \\[1mm]
& \myspace\myspace\myspace\textrm{when $x,z\in \cD(A_P)$.}
\end{split}
\end{equation}

\end{theorem}

\begin{proof}
The proof splits into two parts. First, we establish the validity of \eqref{e:riccati-eq} 
for $x,z\in \cD(A)$. Beside Lemma~\ref{l:riccati-op}, Corollary~\ref{c:differentiability-on-domain_A} will provide the crucial tool.
Next, we show the validity of \eqref{e:riccati-eq} with $x,z\in \cD(A_P)$.
We will use once again Lemma~\ref{l:riccati-op}, along with the intrinsic representation
\eqref{e:dom-a_p-described} of $A_P$ in terms of the operator $\Gamma$.
% as defined by the weak limit \eqref{e:weak-limit_2}.

1. We recall the alternative representation \eqref{e:alternative-formula}
of the Riccati operator $P$ obtained in the previous Lemma, and for $x,z\in Y$ write 
the inner product $(Px,z)_Y$ as a function of $t$: 
\begin{equation} \label{e:follows-from-alternative}
(Px,z)_Y= \int_0^\infty (Re^{A_Pt}x,Re^{At}z)_Z\,dt
=\int_t^\infty (Re^{A_P(\tau-t)}x,Re^{A(\tau-t)}z)_Z\,d\tau.
\end{equation}
Taking now $x,z\in \cD(A)$, in view of \eqref{e:formula-for-derivative} 
of Corollary~\ref{c:differentiability-on-domain_A}, we differentiate both sides 
of the obtained formula with respect to $t$ to find % since $P$ does not depend on $t$
\begin{align}
0 &=\frac{d}{dt}(Px,z)_Y= -(Rx,Rz)_Z 
\nonumber \\[1mm]
& \qquad - \int_t^\infty 
\big(e^{A_P(\tau-t)}A_Px,R^*Re^{A(\tau-t)}z\big)_{[\cD({A^*}^\epsilon)]',\cD({A^*}^\epsilon)}\,d\tau
\label{e:line-0}\\[1mm]
& \qquad - \int_t^\infty \big(Re^{A_P(\tau-t)}x,Re^{A(\tau-t)}Az\big)_Z\,d\tau
\nonumber \\[1mm]
& = -(Rx,Rz)_Z -\int_t^\infty 
\big(e^{A_P(\tau-t)}Ax,R^*Re^{A(\tau-t)}z\big)_Y\,d\tau
\nonumber %\label{e:line-2} 
\\[1mm]
& \qquad +\int_t^\infty 
\big(e^{A_P(\tau-t)}BB^*Px,R^*Re^{A(\tau-t)}z\big)_{[\cD({A^*}^\epsilon)]',\cD({A^*}^\epsilon)}
\,d\tau
\nonumber %\label{e:line-3} 
\\[1mm]
& \qquad 
- \int_t^\infty \big(Re^{A_P(\tau-t)}x,Re^{A(\tau-t)}Az\big)_Z\,d\tau
\nonumber %\label{e:line-4} 
\\[1mm]
& = -(Rx,Rz)_Z-(Ax,Pz)_Y + (B^*Px,B^*Pz)_U-(Px,Az)_Y\,,
\nonumber %\label{e:line-5} 
\end{align}
so that 
\begin{equation} \label{e:line-6} 
(Ax,Pz)_Y + (Px,Az)_Y - (B^*Px,B^*Pz)_U + (Rx,Rz)_Z=0, \; x,z\in\cD(A)\,.
\end{equation}
Note carefully that the integrand in % the right hand side of 
\eqref{e:line-0}
makes sense {\em a priori} as a duality pairing on $\cD({A^*}^\epsilon)$
(which is summable because of \eqref{e:regularity-of-derivative}),
whereas all the summands in \eqref{e:line-6} actually make sense as inner products in the 
spaces $Y$, $U$ or $Z$, since in view of Theorem~\ref{t:bounded-gain} $B^*P$ 
is bounded ({\em a fortiori}) on $\cD(A)$. 
Therefore, $P$ solves the Algebraic Riccati Equation \eqref{e:riccati-eq} on $\cD(A)$.

\smallskip
2. We preliminarly recall that if $x\in \cD(A_P)$, then $A_Px = A\,(x+\Gamma x)$, where 
the operator $\Gamma$---which is defined by the equivalent weak limits \eqref{e:weak-limit_1} and \eqref{e:weak-limit_2}---{\em coincides} on $\cD(A_P)$ with $-A^{-1}B\,B^*P$; see
Remark~\ref{r:operator-gamma}.
% In fact, we recall that in view of Lemma~\ref{l:riccati-op} $x\in \cD(A_P)$ ensures 
% $P\Phi(\cdot)x \in C([0,\infty),\cD(A^*))$, which in turn implies 
% $B^*P\Phi(\cdot)x \in C([0,\infty),U)$.
% Hence, the weak limit in \eqref{e:weak-limit_2} is strong, yielding 
% $\Gamma \equiv -A^{-1}B\,B^*P$ on $\cD(A_P)$.

Given now $x,z \in \cD(A_P)$, we compute 
\begin{align} \label{e:willbe-are}
& (A^*Px,z)_Y + (x,A^*Pz)_Y =\textrm{(by Lemma~\ref{l:riccati-op})}
\nonumber \\
& \qquad = -(Rx,Rz)_Y - (PA_Px,z)_Y + (x,A^*Pz)_Y =\nonumber\\
& \qquad= -(Rx,Rz)_Y - (PA(x+\Gamma x),z)_Y + (x,A^*Pz)_Y =\nonumber\\
& \qquad= -(Rx,Rz)_Y - (x+\Gamma x,A^*Pz)_Y + (x,A^*Pz)_Y =\nonumber\\ 
& \qquad= -(Rx,Rz)_Y - (\Gamma x,A^*Pz)_U\,.
\end{align}
On the other hand, we readily have 
\begin{equation} \label{e:quadratic-term} 
\begin{split}
(\Gamma x,A^*Pz)_Y &=-(A^{-1}B\,B^*Px,A^*Pz )_Y
\\
& = -(B\,B^*Px,Pz )_{[\cD(A^*)]',\cD(A^*)}= -(B^*Px,B^*Pz )_Y\,,
\end{split}
\end{equation}
with $B^*P\in \cL(\cD(A_P),U)$.
Thus, inserting \eqref{e:quadratic-term} in \eqref{e:willbe-are} we find that $P$ satisfies 
the Algebraic Riccati equation \eqref{e:riccati-eq_2} for any $x,z\in \cD(A_P)$,
with $B^*P\in \cL(\cD(A_P),U)$, thus concluding the proof. 
\end{proof}

%-----------------------------------------------------------------------------------------
\section{Illustrations} \label{s:examples}
%-----------------------------------------------------------------------------------------
In this section we give a significant illustration of the applicability of the 
infinite time horizon optimal control theory provided by Theorem~\ref{t:main}.
The boundary control problem under examination is the thermoelastic plate model 
studied in \cite{bucci-las-thermo} and \cite{abl-1}. 
Recall that this specific PDE problem constituted the prime motivation for the introduction 
of the novel class of control systems characterized by the abstract 
assumptions listed in \cite[Hypotheses~2.2]{abl-2}.

%-----------------------------------------------------------------------------------------
\subsection{A thermoelastic plate model with boundary thermal control}
%-----------------------------------------------------------------------------------------
We consider a classical (linear) PDE model for the determination of displacements and the temperature distribution in a thin plate; see \cite{lagnese-lions,lagnese}.
Let $\Omega$ be a bounded domain of $\mathbb{R}^2$, with smooth boundary 
$\Gamma$.
The PDE system comprises a Kirchhoff elastic equation for the vertical displacement 
$w(x,t)$ of the plate and the heat equation for the temperature distribution $\theta(x,t)$.
% the coupling is full interior.
%
The plate equation is supplemented with {\em clamped} boundary conditions, whereas a control action on the temperature, represented by the function $u(x,t)$, is exercised through $\Gamma$. 
Thus, the PDE problem reads as follows (the constant $\rho$ is positive, $\nu$ denotes the unit outward normal to the curve $\Gamma$):
\begin{equation} \label{e:thermo-system_0}
\begin{cases}
w_{tt} - \rho \Delta w_{tt} + \Delta^2 w + \Delta\theta=0 
&  \mbox{in } \,\Omega \times (0,\infty)
\\[1mm] 
\theta_t - \Delta \theta - \Delta w_t = 0
& \mbox{in } \,\Omega \times (0,\infty)
\\[1mm] 
w=\frac{\partial w}{\partial\nu}=0 
\quad % \textrm{\small (clamped B.C.)} 
& \mbox{on } \Gamma \times (0,\infty)
\\[1mm] 
\theta = u % \qquad \textrm{\small (Dirichlet boundary control)} 
& \mbox{on } \Gamma \times (0,\infty)
\\[1mm]
w(0,\cdot) = w^0, \; w_t(0,\cdot) = w^1; \quad
\theta(0,\cdot) = \theta^0 & \mbox{in } \Omega\,.
\end{cases}
\end{equation}
With \eqref{e:thermo-system_0} we associate the natural quadratic functional 
\begin{equation} \label{e:thermo-functional}
\int_0^\infty\!\!\!\int_\Omega 
\big(|\Delta w(x,t)|^2+ |\nabla w_t(x,t)|^2+|\theta(x,t)|^2\big)\, dx\,dt
+\,\int_0^\infty\!\!\!\int_\Gamma |u(x,t)|^2 \, ds\,dt
\end{equation}
to be minimized overall $u\in L^2(\Gamma \times (0,\infty))$, where $w$ solves
the boundary control problem \eqref{e:thermo-system_0}.
Note that when $u\equiv 0$ the functional \eqref{e:thermo-functional} is nothing but the integral 
over $(0,\infty)$ of the physical energy $E(t)$ of the system.

We recall that the finite time horizon optimal control problem for the controlled PDE system 
\eqref{e:thermo-system_0} was first studied in \cite{bucci-las-thermo}, and then 
fully solved according to the novel abstract theory set forth in \cite{abl-2}.
The preliminary PDE analysis carried out in \cite{bucci-las-thermo}, combined with the key boundary regularity result established in \cite[Theorem~1.1.]{abl-1} provided the proof.
\\
Here we aim to complete the study of the associated optimal control problems including the 
infinite time horizon case.
Specifically, we will show that the model under investigation fits as well into the abstract 
framework designed by Assumptions~\ref{h:ip1} and \ref{h:ip2}, thereby ensuring the applicability of Theorem~\ref{t:main}.

\smallskip
We already know that the boundary control problem \eqref{e:thermo-system_0} can be recast as an 
abstract control system of the form \eqref{e:state-eq} in the state variable 
$y=(w,w_t,\theta)$, with appropriate (and explicitly derived) dynamics and control 
operators $(A,B)$; see \cite[\S~2]{bucci-las-thermo} for all details.
A thorough analysis of the semigroup formulation of the uncontrolled problem is provided
by \cite{las-trig-thermo}, where the predominant hyperbolic character of the coupled 
PDE system (in the case $\rho>0$) was first pointed out.
The state and control spaces are given by 
\begin{equation*}
Y= H^2_0(\Omega)\times H^1_0(\Omega)\times L^2(\Omega)\,, \quad\qquad U=L^2(\Gamma)\,,
\end{equation*}
respectively. 
The plan is thus to check the complex of requirements contained in Assumptions~\ref{h:ip1}--\ref{h:ip2}.

\smallskip
\noindent
{\bf Verification of Assumption~\ref{h:ip1}. }
The validity of the basic Assumption~\ref{h:ip1} has been already discussed in 
\cite{bucci-las-thermo}, yielding the explicit statements of 
\cite[Proposition~2.1.]{bucci-las-thermo}.
We recall from \cite[Remark~2.2]{bucci-las-thermo} that well-posedness of the uncontrolled model 
was proved in \cite{las-trig-thermo}, while an easy computation provides boundedness of the linear operator $A^{-1}B$; see (2.23) in \cite{bucci-las-thermo}. 
Instead, the exponential stability of the underlying semigroup---by far a more challenging
issue---was established in \cite{avalos-las-trieste}.

\smallskip
\noindent
{\bf Verification of Assumptions~\ref{h:ip2}. }
We must verify that all the requirements listed in Assumption~\ref{h:ip2} are fulfilled.
Accordingly, we recall from \cite[\S~5]{abl-1} that given $z_0=(w^0,w^1,\theta^0)\in \cD(A^*)$,
one has
\begin{equation}
B^*e^{A^*t}z_0= \frac{\partial\theta}{\partial\nu}\Big|_\Gamma\,,
\end{equation}  
where now $\theta(x,t)$ is the thermal component of the solution $y(t)=(w(t),w_t(t),\theta(t))$ 
to an initial/boundary value problem which comprises the same thermoelastic system of 
\eqref{e:thermo-system_0}, yet with {\em homogeneous} boundary conditions and with a slightly 
different initial condition, that is $y(0)=(w^0,-w^1,\theta^0)$. 
The change of sign on a component of initial data is not influential, and it justifies the estimates performed on the solution to the original PDE problem \eqref{e:thermo-system_0} 
with $u\equiv 0$.
% (notice that $\|(w_0,-w_1,\theta_0)\|=\|(w_0,-w_1,\theta_0)\| )

For the reader's convenience we record from \cite[\S~5]{abl-1} the essential steps of the computations leading to the sought decomposition \eqref{e:key-hypo} of $B^*e^{A^*t}y_0$, with suitable $F$ and $G$ that will be shown to satisfy the series of assumptions in \eqref{h:ip2}.

\noindent
{\bf 1. Notation. }
In the formulas below the symbol $-A_D$ denotes the (unbounded) linear operator which is the 
realization of the Laplace operator $\Delta$ in $H=L^2(\Omega)$, when supplemented with
(homogeneous) Dirichlet boundary condition, i.e. 
\begin{equation*}
A_Dw= -\Delta w\,, \quad w\in \cD(A_D)=H^2(\Omega)\cap H^1_0(\Omega)\,.
\end{equation*}
It is well known that $-A_D$ is the generator of a strongly continuous semigroup $e^{-A_Dt}$
in $H$, which moreover is analytic. 
The fractional powers $A_D^\alpha$ are well defined for all $\alpha\in (0,1)$, and there exist a
positive constant $\omega_L$ and constants $L_\alpha\ge 1$ such that the following estimates hold true:
\begin{equation} \label{e:parabolic-estimates}
\|A_D^{\alpha}e^{-A_Dt}\|_{\cL(H)}\le L_\alpha \frac{e^{-\omega_L t}}{t^\alpha}\qquad \alpha\in [0,1]\,.
\end{equation} 
The related (positive) operator $\cM=I+\rho A_D$ is employed in the abstract formulation 
of the elastic equation and hence will occur in the computations below.
Instead, $\cA$ will denote the realization of the bilaplacian $\Delta^2$ in $\Omega$
with {\em homogeneous} clamped boundary conditions.

Finally, let $D$ be the map which associates to any function in $L^2(\Gamma)$ its harmonic 
extension in $\Omega$. 
Classical trace theory (\cite{lions-magenes}) yields
\begin{equation*}
D \; \textrm{continuous\,:} \; L^2(\Gamma) \longrightarrow H^{1/2}(\Omega)\subset \cD(A_D^{1/4-\sigma})\,, 
\quad 0<\sigma< \frac14\,,
\end{equation*} 
% according to the identification of Sobolev spaces with appropriate fractional powers of the 
% operator $A_D$,
which implies
\begin{equation}\label{e:fractional-bound}
A_D^{1/4-\sigma}D \quad \textrm{continuous\,:} \; L^2(\Gamma) \longrightarrow L^2(\Omega)\,, 
\quad 0<\sigma< \frac14\,.
\end{equation} 
In addition, the following well known result will be used throughout: 
\begin{equation}\label{e:d^*a}
D^*A_Dh= \frac{\partial h}{\partial\nu}\Big|_{\Gamma}\,, 
\qquad h\in H^{3/2+\sigma}(\Omega)\cap H^1_0(\Omega)\,,
\quad \sigma>0\,;
\end{equation}  
see, e.g., \cite[Lemma~3.1.1, p.~181]{las-trig-encyclopedia}.
% easy application of Green's second theorem

\medskip
\noindent
{\bf 2. Explicit decomposition. }
We take the explicit expression of $\theta(t)$ and utilize \eqref{e:d^*a} 
to rewrite 
\begin{equation*}
\frac{\partial\theta}{\partial\nu}\equiv D^*A_D \theta(t)
= D^*A_D \Big[ e^{-A_Dt} \theta^0 - A_D\int_0^t e^{-A_D(t-s)}w_t(s)\,ds\Big]\,;
\end{equation*}
we next integrate by parts, obtaining first (since according to the clamped boundary conditions
$\frac{\partial w_t}{\partial\nu}=0$ on $\Gamma$): 
\begin{equation*}
\begin{split}
\frac{\partial\theta}{\partial\nu} 
& = \underbrace{D^*A_D e^{-A_Dt} \theta^0}_{F_1(t)y_0} 
+ \underbrace{D^*A_D e^{-A_Dt} w^1}_{F_2(t)y_0}
- \underbrace{D^*A_D\int_0^t e^{-A_D(t-s)}w_{tt}(s)\,ds}_{\psi(t;y_0)}.
\end{split}
\end{equation*}
Thus, we utilize the elastic equation $w_{tt}=-\cM^{-1}\cA w(s)+\cM^{-1}A_D \theta(s)$
to find 
\begin{equation}
\begin{aligned}
\psi(t;y_0) &= \underbrace{-D^*A_D\int_0^t e^{-A_D(t-s)}\cM^{-1}\cA w(s)\,ds}_{\psi_1(t;y_0)} 
\\ &\qquad\qquad
+ \underbrace{D^*A_D\int_0^t e^{-A_D(t-s)}\cM^{-1}A_D\theta(s)\,ds}_{-F_3(t)y_0}
\\
& = \underbrace{D^*A_D\int_0^t e^{-A_D(t-s)}\cM^{-1}A_D \Delta w(s)\,ds}_{\psi_{11}(t;y_0)}
\\  &\qquad\qquad
\underbrace{-D^*A_D\int_0^t e^{-A_D(t-s)}\cM^{-1}A_D\,\Delta w(s)\big|_{\Gamma}\,ds}_{\psi_{12}(t;y_0)}- F_3(t;y_0)
\end{aligned}
\end{equation}
where the further splitting of $\psi_1(t;y_0)=\psi_{11}(t;y_0)+\psi_{12}(t;y_0)$
is a consequence of 
\begin{equation*}
\cM^{-1}\cA w= -\cM^{-1} A_D\big(\Delta w-D\Delta w|_{\Gamma}\big)\,;
\end{equation*}
see \cite[\S5, formula (5.9)]{bucci-las-thermo}.
We will see that the integral $\psi_{11}(t;y_0)$ eventually contribute to the term $F(t)y_0$, while it was shown in \cite{bucci-las-thermo} that $\psi_{12}(t;y_0)$ does satisfy Assumption~\ref{h:ip2}(ii) and (iii)(a); hence, $\psi_{12}(t;y_0)$ is identified with $G(t)y_0$.

Summarizing, we found the following decomposition: 
\begin{equation*}
\frac{\partial\theta}{\partial\nu}= \sum_{i=1}^4 F_i(t)y_0 + G(t)y_0\,,
\end{equation*}
where 
\begin{align*}
F_1(t)y_0 &= D^*A_D e^{-A_Dt} \theta^0\,,
\quad
F_2(t)y_0 = D^*A_D e^{-A_Dt} w^1
% D^*A_D^{1/2} e^{-A_Dt} A_D^{1/2}w^1\,,
\\
F_3(t)y_0 &= -D^*A_D\int_0^t e^{-A_D(t-s)}\cM^{-1}A_D\theta(s)\,ds\,,
\\
F_4(t)y_0 &=D^*A_D\int_0^t e^{-A_D(t-s)}\cM^{-1}A_D \Delta w(s)\,ds\,,
\end{align*}
and 
\begin{equation}\label{e:component-g}
G(t)y_0= D^*A_D\int_0^t e^{-A_D(t-s)}\cM^{-1}A_D\,\Delta w(s)\big|_{\Gamma}\,ds\,.
\end{equation}

\medskip
\noindent
{\bf 3. Estimates. }
We proceed to check the validity of the first among Assumptions~\ref{h:ip2} on either term
$F_i(t)y_0$, $i=1,\dots,4$.
Employing the various properties recorded above---in particular, using repeatedly 
the analytic estimates \eqref{e:parabolic-estimates}---and that 
\begin{equation*}
\|(w(t),w_t(t),\theta(t))\|_Y=\|e^{At}(w^0,-w^1,\theta^0)\|_Y\le Me^{-\omega t}\|y_0\|_Y\,,
\end{equation*}
we find 
%\begin{subequations}
\begin{align}
\|F_1(t)y_0\| & % = \|D^*A_D^{1/4-\sigma}\, A_D^{3/4+\sigma}e^{-A_Dt} \theta^0\|
\le \|D^*A_D^{1/4-\sigma}\|\,L_{3/4+\sigma}\frac{e^{-\omega_L t}}{t^{3/4+\sigma}}\|\theta^0\|
\le C_{1,\sigma}\,t^{-3/4-\sigma} e^{-\omega_L t}\|y_0\|\,;
\label{e:bound-f_1}\\[2mm]
\|F_2(t)y_0\| & = \|D^*A_D^{1/2}\|\,\|e^{-A_Dt} A_D^{1/2}w^1\|
\nonumber \\[1mm]
& \le 
\|D^*A_D^{1/4-\sigma}\|\,L_{1/4+\sigma}\frac{e^{-\omega_L t}}{t^{1/4+\sigma}}\|A_D^{1/2}w^1\|
= C_{2,\sigma}\,t^{-1/4-\sigma}e^{-\omega_L t}\|w^1\|_{1,\Omega}
\nonumber \\[1mm]
& \le C_{2,\sigma}\,e^{-\omega_L t}\max\{t^{-3/4-\sigma},1\}\|y_0\|\,;
\label{e:bound-f_2}\\[2mm]
\|F_3(t)y_0\| & \le \|D^*A_D^{1/4-\sigma}\| L_{3/4+\sigma}\,
\int_0^t\frac{e^{-\omega_L(t-s)}}{(t-s)^{3/4+\sigma}} \|\cM^{-1}A_D\| \|\theta(s)\|\, ds
\nonumber \\[1mm]
& \le \|D^*A_D^{1/4-\sigma}\| L_{3/4+\sigma}\,\|\cM^{-1}A_D\|
\int_0^t\frac{e^{-\omega_L(t-s)}}{(t-s)^{3/4+\sigma}}\,M e^{-\omega s}\|y_0\|\,ds 
\nonumber \\ 
& \le C_{3,\sigma}\, t^{1/4-\sigma} e^{-\overline{\omega} t}\|y_0\|\,, 
\qquad \overline{\omega}=\min\{\omega,\omega_L\}\,;
\label{e:bound-f_3}\\[2mm]
\|F_4(t)y_0\| & \le \|D^*A_D^{1/4-\sigma}\| L_{3/4+\sigma}\,
\int_0^t\frac{e^{-\omega_L(t-s)}}{(t-s)^{3/4+\sigma}}\|\cM^{-1}A_D\| \|w(s)\|_{2,\Omega}\,ds
\nonumber \\[1mm]
& \|D^*A_D^{1/4-\sigma}\| L_{3/4+\sigma}\,
\int_0^t\frac{e^{-\omega_L(t-s)}}{(t-s)^{3/4+\sigma}}\|\cM^{-1}A_D\| \,M e^{-\omega s}\|y_0\|\,ds
\nonumber \\[1mm]
& \le C_{4,\sigma}\,t^{1/4-\sigma} e^{-\overline{\omega} t}\|y_0\|\,,
\qquad \overline{\omega}=\min\{\omega,\omega_L\}\,.
\label{e:bound-f_4}
\end{align}
%\end{subequations}
Thus, if $\sigma$ is fixed, the bounds obtained in \eqref{e:bound-f_3} and \eqref{e:bound-f_4}
imply that 
\begin{equation}\label{e:bound-f_34}
\exists C, \,\eta>0: \quad \|F_i(t)y_0\| \le C\,e^{-\eta t}\|y_0\|\,,
\qquad i=1,2\,.
\end{equation}
(We note that, more precisely, for any $\eta<\min\{\omega,\omega_L\}$ there exists a constant 
$C_\eta>0$ such that the estimate in \eqref{e:bound-f_34} holds.
However, the formulation \eqref{e:bound-f_34} will suffice.)

Consequently, with a fixed $\sigma$ the estimates \eqref{e:bound-f_1} and \eqref{e:bound-f_2} combined with \eqref{e:bound-f_34} show that Assumption~\ref{h:ip2}(i) is satisfied with 
$\gamma>3/4$ and $\eta<\min\{\omega,\omega_L\}$.

\smallskip
% As it was observed above, 
Finally, that the component $G(t)y_0$ in \eqref{e:component-g} fulfils the requirements of Assumptions~\ref{h:ip2}(ii) and (iii)(a) was proved in \cite{bucci-las-thermo}.
The more challenging Assumption~\ref{h:ip2}(iii)---whose PDE counterpart is a suitable 
regularity result for the boundary traces of $\theta_t$---has been established in \cite{abl-1}.   

Therefore, provided the observation operator $R$ meets Assumption~\ref{h:ip2}(iii)(b),
all the hypotheses of Theorem~\ref{t:main} are satisfied; consequently, all its statements follow.
Indeed, in the present case the observation operator which occurs in the functional 
\eqref{e:thermo-functional} is the {\em identity}, % i.e.~$R=I$, 
and it is not difficult to ascertain that Assumption~\ref{h:ip2}(iii)(b) holds true
if $\epsilon$ is taken sufficiently small; 
see \cite[\S 3]{bucci-las-thermo} and \cite[Remark~2.5]{abl-2}.
   
% APPENDIX
\appendix
%-----------------------------------------------------------------------------
\section{Appendix} \label{a:lonely}
%-----------------------------------------------------------------------------
%\subsection{Further regularity properties of the optimal pair}
%
\begin{proof}[Proof of Proposition~\ref{p:statement-eight}]
We shall use the representation \eqref{e:phi} for $\Phi(t)x$.
Let $x\in\cD(A^\epsilon)$, with $\epsilon$ given by Assumption~\ref{h:ip2}(iii).
By standard semigroup theory we know that for any $\delta<\omega$, 
\begin{equation} \label{e:reg-start}
e^{A\cdot}x \in L_{\delta}^p(0,\infty;\cD(A^\epsilon)) \qquad \forall p\in [1,\infty]\,;
\end{equation}
then, owing to Assumption~\ref{h:ip2}(iii)(b), we obtain as well
\begin{equation*}
R^*Re^{A\cdot}x \in L_{\delta}^p(0,\infty;\cD({A^*}^\epsilon)) \qquad \forall p\in [1,\infty]\,.
\end{equation*}
In particular, $R^*Re^{A\cdot}x \in L_{\delta}^p(0,\infty;\cD({A^*}^\epsilon))$ with 
$p>1/(1-\gamma)$ and Proposition~\ref{p:stimeL*}(iv) establishes
\begin{equation*}
L^*R^*Re^{A\cdot}x \in L_{\delta}^\infty(0,\infty;U)\,.
\end{equation*}

We now make the following Claim.
% CLAIM
\begin{claim} \label{claim-lambda}
The operator $\Lambda = I + L^*R^*RL$ admits a bounded inverse
\begin{equation*}
\Lambda^{-1}: L^\infty_\delta(0,\infty;U) 
\longrightarrow L^\infty_{\delta-\sigma_0}(0,\infty;U)\,,
\end{equation*}
with appropriate $\sigma_0\in (0,\delta)$.
\end{claim}
% END of CLAIM

\smallskip
\noindent
Assuming that Claim~\ref{claim-lambda} is valid, it follows that 
\begin{equation*}
\Lambda^{-1} L^*R^*Re^{A\cdot}x \in L_{\theta}^\infty(0,\infty;U)
\quad \textrm{for some $\theta<\delta<\omega\wedge\eta$.}
\end{equation*}
Thus, using now Proposition~\ref{p:stimeL}(v) we see that 
\begin{equation} \label{e:reg-end}
L\Lambda^{-1} L^*R^*Re^{A\cdot}x \in L_{\theta}^\infty(0,\infty;\cD(A^\epsilon))
\subset L_{\beta}^p(0,\infty;\cD(A^\epsilon))
\end{equation}
for any $\beta$ such that $0<\beta<\theta$ and for all $p\in [2,\infty]$.

The regularity in \eqref{e:reg-end}, combined with the one in \eqref{e:reg-start},
shows that \eqref{e:goal_1} holds true.
Then, using the boundedness of the gain operator $B^*P$ on $\cD(A^\epsilon)$
in the feedback representation \eqref{e:feedback} of the optimal control, 
we easily obtain that \eqref{e:goal_1} implies \eqref{e:goal_2}, thus completing the proof.
\end{proof}

\begin{proof}[Proof of Claim~\ref{claim-lambda}]
We proceed pretty much in the same way as in the proof of Theorem~\ref{t:invertible-on-X_q}.
The tools which play a major role are, once more, the smoothing properties
of the operators $L$ and $L^*$, as well as the inclusions  
$L_{\alpha}^r(0,\infty;U)\subset L_{\beta}^s(0,\infty;U)$ for $\beta\in (0,\alpha)$
and $r<s$ (including $s=+\infty$).

\smallskip
\noindent
1. We seek to solve uniquely the equation
\begin{equation} \label{e:goal-2}
g + L^*R^*RLg = h\in L_{\delta}^\infty(0,\infty;U)\,,
\end{equation}
with arbitrary $\delta \in (0,\omega\wedge\eta)$.
An elementary calculation shows that $h\in L_{\theta}^2(0,\infty;U)$ as well, 
for any $0<\theta <\delta$.
% see \eqref{e:elementary-split} 
%
Thus, since by Lemma~\ref{l:invertible-on-L2delta} $\Lambda$ is boundedly invertible in 
$L_{\theta}^2(0,\infty;U)$ provided $\theta$ is taken sufficiently small, 
there exists a unique function $g\in L_{\theta}^2(0,\infty;U)$ such that 
\eqref{e:goal-2} is satisfied.
We will show that---possibly taking a smaller $\theta$---in fact 
\begin{equation} \label{e:goal-3}
g\in L_{\theta}^\infty(0,\infty;U)\,.
\end{equation}

\smallskip
\noindent
2. Preliminarly, we prove that 
\begin{equation} \label{e:intermediate-goal-2}
\exists n_1\in \mathbb{N}: \quad  (L^*R^*RL)^{n_1} \; 
\textrm{continuous: $L_{\theta}^2(0,\infty;U)\to L_{\theta}^\infty(0,\infty;U)$}\,.
\end{equation}
Let $f\in L_{\theta}^2(0,\infty;U)$. 
Since the successive application of the operators $L$ and $L^*$ improve the
(time regularity) summability exponents, there exists an integer $\overline{n}$ such that
\begin{equation*} 
(L^*R^*RL)^{\overline{n}}f\in L_{\theta}^{q'}(0,\infty;U)\,,
\end{equation*}
which in view of Proposition~\ref{p:stimeL}(v) implies 
$e^{\theta\cdot}L(L^*R^*RL)^{\overline{n}}f\in C_b([0,\infty);\cD(A^\epsilon))$.
Then, $e^{\theta\cdot}R^*RL(L^*R^*RL)^{\overline{n}}f\in C_b([0,\infty);\cD({A^*}^\epsilon))$
and by Proposition~\ref{p:stimeL*}(v)
\begin{equation*} 
(L^*R^*RL)^{\overline{n}+1}f= L^*R^*RL(L^*R^*RL)^{\overline{n}}f
\in L_{\theta}^\infty(0,\infty;U)\,.
\end{equation*}
Thus, \eqref{e:intermediate-goal-2} is satisfied with $n_1= \overline{n}+1$.

\smallskip
\noindent
3. We return to \eqref{e:goal-2} and argue as follows. 
The solution $g$ of \eqref{e:goal-2} is such that
\begin{equation}\label{e:goal-2-equivalent} 
g = h-L^*R^*RLg\,;
\end{equation}
then, if $n_1=1$ in \eqref{e:intermediate-goal-2}, since $h$ and $L^*R^*RLg$ in
\eqref{e:goal-2-equivalent} both belong to $L_{\theta}^\infty(0,\infty;U)$, we immediately
obtain \eqref{e:goal-3}.
The case $n_1>1$ is treated as follows.

We apply the operator $L^*R^*RL^*$ to both members of \eqref{e:goal-2-equivalent},
thus obtaining  
\begin{equation} \label{e:passo-2}
L^*R^*RL^*g  = L^*R^*RL^*h-(L^*R^*RL^*)^2 g\,.
\end{equation}
Observe now that $L^*R^*RL^*h\in L_{\theta}^\infty(0,\infty;U)$, as the given regularity of $h$ is maintained by the application of $L^*R^*RL$; if, in addition, $n_1=2$ in 
\eqref{e:intermediate-goal-2}, the identity \eqref{e:passo-2} yields 
$L^*R^*RL^*g\in L_{\theta}^\infty(0,\infty;U)$.
Using the obtained regularity for $L^*R^*RL^*g$ in \eqref{e:goal-2-equivalent},
we find that $g\in L_{\theta}^\infty(0,\infty;U)$, that is \eqref{e:goal-3}.

If, instead, $n_1>2$, inserting \eqref{e:passo-2} into \eqref{e:goal-2-equivalent} we find
the novel identity
\begin{equation} \label{e:passo-3}
g=h-L^*R^*RL^*h+(L^*R^*RL^*)^2 g\,.
\end{equation}
To pinpoint the regularity of the summand $(L^*R^*RL^*)^2 g$, we return to 
\eqref{e:passo-2} and apply the operator $L^*R^*RL^*$ to both members. 
Plugging the obtained expression for $(L^*R^*RL^*)^2 g$ into \eqref{e:passo-3} yields 
\begin{equation*}
g=\sum_{j=0}^2 (-1)^j (L^*R^*RL^*)^j h+ (L^*R^*RL^*)^3 g\,.
\end{equation*}
The iteration of this argument eventually yields  
\begin{equation*}
g=\sum_{j=0}^{n_1-1} (-1)^j (L^*R^*RL^*)^j h+ (L^*R^*RL^*)^{n_1} g\,,
\end{equation*}
where both summands on the right hand side belong to $L_{\theta}^\infty(0,\infty;U)$,
which establishes \eqref{e:goal-3} and concludes the proof.

\end{proof}

% References

\end{document}